\documentclass[11pt,a4paper,leqno]{amsart}
\usepackage[T1]{fontenc}
\usepackage[utf8]{inputenc}

\usepackage{lmodern}
\usepackage{microtype}

\usepackage{amsthm}
\usepackage{amsmath}
\usepackage{amsfonts}
\usepackage{amssymb}
\usepackage{enumitem}

\usepackage{esint}
\usepackage[vmargin=3.5cm]{geometry}

\usepackage[unicode=true, pdfusetitle]{hyperref}
\let\C\undefined

\usepackage[abbrev,backrefs]{amsrefs}
\usepackage[capitalize]{cleveref}

\newtheorem{proposition}{Proposition}[section]
\newtheorem{theorem}[proposition]{Theorem}
\newtheorem{lemma}[proposition]{Lemma}

\theoremstyle{definition}
\newtheorem{definition}[proposition]{Definition}

\newcommand{\defeq}{\coloneqq}

\newcommand{\Nset}{\mathbb{N}}
\newcommand{\Zset}{\mathbb{Z}}
\newcommand{\Rset}{\mathbb{R}}

\newcommand{\Sset}{\mathbb{S}}

\newcommand{\Bset}{\mathbb{B}}
\newcommand{\Hset}{\mathbb{H}}

\newcommand{\dif}{\,\mathrm{d}}
\usepackage{mathtools}

\newcommand{\compose}{\,\circ\,}
\newcommand{\manifold}[1]{\mathcal{#1}}

\DeclarePairedDelimiter{\brk}{(}{)}
\DeclarePairedDelimiter{\abs}{\lvert}{\rvert}
\DeclarePairedDelimiter{\norm}{\lVert}{\rVert}

\DeclarePairedDelimiter{\floor}{\lfloor}{\rfloor}

\DeclarePairedDelimiterX{\intvc}[2]{[}{]}{#1,#2}
\DeclarePairedDelimiterX{\intvl}[2]{(}{]}{#1,#2}
\DeclarePairedDelimiterX{\intvr}[2]{[}{)}{#1,#2}
\DeclarePairedDelimiterX{\intvo}[2]{(}{)}{#1,#2}
\DeclarePairedDelimiterX{\setcond}[2]{\{}{\}}{#1 \,\delimsize\vert\, #2}

\newcommand{\restr}[1]{\!\upharpoonright_{#1}}
\newcommand{\Deriv}{\mathrm{D}}
\newcommand{\meanosc}{\mathrm{MO}}

\providecommand{\st}{\,\vert\,}
\newcommand\stSymbol[1][]{%
\nonscript\;#1\vert
\allowbreak
\nonscript\;
\mathopen{}}
\DeclarePairedDelimiterX\set[1]\{\}{%
\renewcommand\st{\stSymbol[\delimsize]}
#1
}

\DeclareMathOperator{\dist}{dist}
\DeclareMathOperator{\supp}{supp}

\DeclareMathOperator{\diam}{diam}
\DeclareMathOperator{\tr}{tr}

\usepackage{constants}

\renewcommand{\PrintDOI}[1]{%
  \href{http://dx.doi.org/#1}{doi:#1}%
}
\BibSpec{book}{%
    +{}  {\PrintPrimary}                {transition}
    +{,} { \textit}                     {title}
    +{.} { }                            {part}
    +{:} { \textit}                     {subtitle}
    +{,} { \PrintEdition}               {edition}
    +{}  { \PrintEditorsB}              {editor}
    +{,} { \PrintTranslatorsC}          {translator}
    +{,} { \PrintContributions}         {contribution}
    +{,} { }                            {series}
    +{,} { \voltext}                    {volume}
    +{,} { }                            {publisher}
    +{,} { }                            {organization}
    +{,} { }                            {address} 
    +{,} { \PrintDateB}                 {date}
    +{,} { }                            {status}
    +{,} { \PrintDOI}                   {doi}
    +{}  { \parenthesize}               {language}
    +{}  { \PrintTranslation}           {translation}
    +{;} { \PrintReprint}               {reprint}
    +{.} { }                            {note}
    +{.} {}                             {transition}
    +{}  {\SentenceSpace \PrintReviews} {review}
}

\title[Extension of critical Sobolev mappings]
{%
Singular extension of critical Sobolev mappings
under an exponential weak-type estimate%
}

\author{Bohdan Bulanyi}

\author{Jean Van Schaftingen}


\begin{document}

\address{
Universit\`{a} di Bologna, Dipartimento di Matematica, Piazza di Porta S. Donato 5, 40126 Bologna, Italy}

\email{Bohdan.Bulanyi@unibo.it}

\address{
Universit\'e catholique de Louvain, Institut de Recherche en Math\'ematique et Physique, Chemin du Cyclotron 2 bte L7.01.01, 1348 Louvain-la-Neuve, Belgium}
\email{Jean.VanSchaftingen@UCLouvain.be}

\begin{abstract}
Given $m \in \mathbb{N} \setminus \{0\}$ and a compact Riemannian manifold $\mathcal{N}$, we construct for every map $u$ in the critical Sobolev space $W^{m/(m + 1), m + 1} (\mathbb{S}^m, \mathcal{N})$,
a map $U : \mathbb{B}^{m + 1}_{1} \to \mathcal{N}$ whose trace is $u$ and which satisfies an exponential weak-type Sobolev estimate.
The result and its proof carry on to the extension to a half-space of maps on its boundary hyperplane and to the extension to the hyperbolic space of maps on its boundary sphere at infinity.
\end{abstract}

\thanks{The authors were supported by the Projet de Recherche T.0229.21 ``Singular Harmonic Maps and Asymptotics of Ginzburg--Landau Relaxations'' of the
Fonds de la Recherche Scientifique--FNRS}

\maketitle

\tableofcontents

\section{Introduction}

The classical trace theory in Sobolev spaces \cite{Gagliardo_1957} states that for every  \(m \in \Nset\setminus \set{0}\) and \(p \in \intvo{1}{\infty}\), there is a well-defined surjective trace operator \(\tr_{\Sset^{m}} : W^{1, p} (\Bset^{m + 1}_{1}, \Rset)\to W^{1-1/p, p} (\Sset^{m}, \Rset)\) that coincides with the restriction on the subsets of continuous functions, where the first-order Sobolev space on the unit ball \(\Bset^{m + 1}_{1} \subset \Rset^{m + 1}\) is defined as 
\begin{equation*} 
 W^{1, p} (\Bset^{m + 1}_{1}, \Rset)
 \defeq \set[\Big]{U : \Bset^{m + 1}_{1} \to \Rset \st U \text{ is weakly differentiable and }
 \int_{\Bset^{m + 1}_{1}} \abs{\Deriv U}^p < \infty}
\end{equation*}
and the fractional Sobolev space \(W^{s, p} (\Sset^m, \Rset)\) on the unit sphere \(\Sset^m\defeq \partial \Bset_{1}^{m+1}\) is defined for \(s \in \intvo{0}{1}\) as
\begin{equation}\label{def of fractsobspace}
 W^{s, p} (\Sset^{m}, \Rset)
 \defeq \set[\bigg]{u : \Sset^{m} \to \Rset \st u \text{ is measurable and }
 \smashoperator{\iint_{\Sset^{m}\times \Sset^{m}}} \frac{\abs{u (x) - u (y)}^p}{\abs{x - y}^{m + sp}}\dif x \dif y < \infty}.
\end{equation}
Moreover, the trace operator \(\tr_{\Sset^{m}}\) has a continuous linear right inverse, so that, in particular, for every function \(u \in W^{1-1/p, p} (\Sset^{m}, \Rset)\) there exists a function \(U \in W^{1, p} (\Bset^{m + 1}_{1}, \Rset) \) such that \(\tr_{\Sset^m} U = u\)  and 
\begin{equation}
\label{eq_ohd8echeen0iuc1aXi8xi5uC}
  \int_{\Bset^{m + 1}_{1}} \abs{\Deriv U}^p
  \le C
  \smashoperator{\iint_{\Sset^{m}\times \Sset^{m}}} \frac{\abs{u (x) - u (y)}^p}{\abs{x - y}^{m + p - 1}}\dif x \dif y,
\end{equation}
where the constant \(C \in \intvo{0}{\infty}\) in \eqref{eq_ohd8echeen0iuc1aXi8xi5uC} depends only on \(m\) and \(p\).

Given a compact Riemannian manifold \(\manifold{N}\), which we can assume without loss of generality to be isometrically embedded into \(\Rset^\nu\) by Nash's isometric embedding theorem \cite{Nash_1956}, we consider the Sobolev spaces of mappings
\begin{equation*}
  W^{1, p} (\Bset^{m + 1}_{1}, \manifold{N})
  \defeq \set{U \in W^{1, p} (\Bset^{m+ 1}_{1}, \Rset^\nu) \st U \in \manifold{N} \text{ almost everywhere in \(\Bset^{m+ 1}_{1}\)}}
\end{equation*}
and
\begin{equation*}
  W^{s, p} (\Sset^{m}, \manifold{N})
  \defeq \set{u \in W^{s, p} (\Sset^{m}, \Rset^\nu) \st u \in \manifold{N} \text{ almost everywhere in \(\Sset^{m}\)}}.
\end{equation*}
Although it follows immediately from the classical trace theory that the trace operator is well defined and continuous from \( W^{1, p} (\Bset^{m + 1}_{1}, \manifold{N})\) to \(W^{1 - 1/p, p} (\Sset^{m}, \manifold{N})\), the proof of the surjectivity fails to extend to this case:
one can just prove that every mapping in \(W^{1 - 1/p, p} (\Sset^m, \manifold{N})\) can be extended to a function in \(W^{1, p} (\Bset^{m + 1}_{1}, \Rset^\nu)\); since the extension is constructed by convolution, there is no hope that the extension would satisfy the nonlinear manifold constraint in general.

In the case of subcritical dimension \(m < p - 1\), maps in the nonlinear Sobolev spaces \(W^{1 - 1/p, p} (\Sset^m, \manifold{N})\) and \(W^{1, p} (\Bset^{m + 1}_{1}, \mathcal{N})\) can be assumed to be continuous by the Sobolev--Morrey embedding; it turns out that all maps in \(W^{1 - 1/p, p} (\Sset^m, \manifold{N})\) are traces of maps in \(W^{1, p} (\Bset^{m + 1}_{1}, \manifold{N})\)
if and only if all continuous maps from \(\Sset^m\) to \(\manifold{N}\) are restrictions of continuous mappings, or equivalently, if and only if the \(m^{\text{th}}\) homotopy group \(\pi_{m} (\manifold{N})\) is trivial: \(\pi_{m} (\manifold{N}) \simeq \set{0}\) \cite{Bethuel_Demengel_1995}*{Th.\ 1}.
For the critical dimension \(m = p - 1\), one gets similarly that all mappings in \(W^{m/(m + 1), m + 1} (\Sset^m, \manifold{N})\) are traces of maps in \(W^{1, m + 1} (\Bset^{m + 1}_{1}, \manifold{N})\) if and only if \(\pi_{m} (\manifold{N}) \simeq \set{0}\) \cite{Bethuel_Demengel_1995}*{Th.\ 2}; this can be explained by the vanishing mean oscillation property of maps in \(W^{m/(m + 1), m + 1} (\Sset^m, \manifold{N})\) and the possibility to extend homotopy and obstruction theories to such maps \cite{Brezis_Nirenberg_1995}.

The situation is radically different in the supercritical dimension case \(m > p - 1\), where it has been proved in a succession of works among other results that
the trace operator is surjective from \( W^{1, p} (\Bset^{m + 1}_{1}, \manifold{N})\) to \( W^{1, 1 - 1/p} (\Sset^{m}, \manifold{N})\)
if and only if the homotopy groups \(\pi_{1} (\manifold{N}), \dotsc, \pi_{\floor{p - 2}} \brk{\manifold{N}}\) are finite and
\(\pi_{\floor{p - 1}} \brk{\manifold{N}}\) is trivial \citelist{\cite{Hardt_Lin_1987}\cite{Bethuel_Demengel_1995}\cite{Bethuel_2014}\cite{Mironescu_VanSchaftingen_2021_Trace}\cite{VanSchaftingen_Extension}}.

Going back to the case of subcritical or critical dimension \(m \le p - 1\), one can wonder whether the \emph{linear} estimate \eqref{eq_ohd8echeen0iuc1aXi8xi5uC} can still hold for the extension of Sobolev mappings.
It has been proved that every map \(u \in W^{1 - 1/p, p} (\Sset^m, \manifold{N})\) has an extension \(U \in W^{1, p} (\Bset^{m + 1}_{1}, \manifold{N})\) with the estimate \eqref{eq_ohd8echeen0iuc1aXi8xi5uC} when either \(m = 1\) and \(\pi_1 (\manifold{N}) \simeq \set{0}\) or \(m \ge 2\), \(\pi_1 (\manifold{N})\) is finite and \(\pi_2 (\manifold{N}) \simeq \dotsb \simeq \pi_{\floor{p}} (\manifold{N}) \simeq \set{0}\) \citelist{\cite{Hardt_Lin_1987}\cite{VanSchaftingen_2020}};
moreover, if there is an extension satisfying \eqref{eq_ohd8echeen0iuc1aXi8xi5uC}, then the homotopy groups \(\pi_{1} (\manifold{N}), \dotsc, \pi_{\floor{p - 1}} (\manifold{N})\) have all to be finite \citelist{\cite{Bethuel_2014}\cite{Mironescu_VanSchaftingen_2021_Trace}};
in the critical case \(m = p - 1\), the additional condition that \(\pi_{m}(\manifold{N}) \simeq \set{0}\) is necessary for a Sobolev estimate on the extension of continuous mappings that have a continuous extension \cite{Mironescu_VanSchaftingen_2021_Trace}.

In view of the obstructions to linear Sobolev estimates on the extension of the form \eqref{eq_ohd8echeen0iuc1aXi8xi5uC}, one can hope to get some \emph{nonlinear} estimates instead.
In the subcritical dimension case, a compactness argument \cite{Petrache_VanSchaftingen_2017}*{Th.~4} shows that given \(m \in \Nset \setminus \set{0}\), \(p > m + 1\) and a compact Riemannian manifold \(\manifold{N}\), there exists a function \(\gamma \in C (\intvr{0}{\infty}, \intvr{0}{\infty})\) such that \(\gamma (0) = 0\), and every map \(u \in W^{1 - 1/p, p} (\Sset^m, \manifold{N})\) that has a continuous extension has an extension \(U \in W^{1, p} (\Bset^{m+1}_{1}, \manifold{N})\) satisfying
\begin{equation}
\label{eq_shaequunah6UShaiSoo5eepi}
 \int_{\Bset^{m + 1}_{1}} \abs{\Deriv U}^p
  \le \gamma\brk[\bigg]{\;
  \smashoperator[r]{\iint_{\Sset^{m}\times \Sset^{m}}} \frac{d(u (x),u (y))^p}{\abs{x - y}^{m + p - 1}}\dif x \dif y},
\end{equation}
where \(d\) is the geodesic distance on \(\manifold{N}\). 

Because of bubbling phenomena, this estimate \eqref{eq_shaequunah6UShaiSoo5eepi} does not go to the endpoint \(p = m + 1\) when \(\pi_{m} (\manifold{N}) \not \simeq\{0\}\). 
Indeed, any map in \(W^{m/(m + 1), m + 1} (\Sset^m, \manifold{N})\) is the weak limit of continuous maps in \(W^{m/(m + 1), m + 1} (\Sset^m, \manifold{N})\) that can be extended so that a weak compactness argument would imply that every map  in \(W^{m/(m + 1), m + 1} (\Sset^m, \manifold{N})\) has an extension \(W^{1, m + 1} (\Bset^{m+1}_{1}, \manifold{N})\), which cannot be the case when \(\pi_{m} (\manifold{N}) \not \simeq \set{0}\) (see \cite{Petrache_Riviere_2014}*{Prop.~ 2.8}).
In this situation, it has been proved in works by Petrache, Rivière and the second author that there is still an extension \(U \in W^{1, (m + 1, \infty)} (\Bset^{m+1}_{1}, \manifold{N})\) satisfying a weak-type Marcinkiewicz--Sobolev estimate \citelist{\cite{Petrache_Riviere_2014}\cite{Petrache_VanSchaftingen_2017}*{Th.\ 2}}: for every \(t \in \intvo{0}{\infty}\),
\begin{equation}
\label{eq_eeyahGie3shee1ahth5luth9}
 t^{m + 1} \mathcal{L}^{m+1} \brk[\big]{\set{x \in \Bset^{m + 1}_{1} \st \abs{\Deriv U(x)} > t }} 
  \le \gamma\brk[\bigg]{\;
  \smashoperator[r]{\iint_{\Sset^{m}\times \Sset^{m}}} \frac{d (u (x), u (y))^{m+1}}{\abs{x - y}^{2 m}}\dif x \dif y}.
\end{equation}

The function \(\gamma\), appearing in the estimate \eqref{eq_eeyahGie3shee1ahth5luth9} for a general target manifold \(\manifold{N}\), is a  wild \emph{double exponential} function \cite{Petrache_VanSchaftingen_2017}*{discussion after Th.\ 2}.
Petrache and Rivière get similar estimates with \(\gamma\) being a polynomial when \(m = 2\) and \(\manifold{N} = \Sset^2\) \cite{Petrache_Riviere_2014}*{Th.~C} (thanks to the Hopf fibration) and an exponential of a power when \(m = 3\) and \(\manifold{N} = \Sset^3\) \cite{Petrache_Riviere_2014}*{Th.~B}.

In the present work, we construct an extension \(U\) that satisfies \eqref{eq_eeyahGie3shee1ahth5luth9}, where \(\gamma\) can be taken to be an exponential function.

\begin{theorem}
\label{theorem_ball}
Let \(m \in \Nset \setminus \set{0}\) and let \(\manifold{N}\) be a compact Riemannian manifold.
There exist constants \(A, B, \delta \in \intvo{0}{\infty}\) such that 
for every \(u \in W^{m/(m + 1), m + 1} (\Sset^m, \manifold{N})\), there
exists a mapping \(U \in W^{1, 1} (\Bset^{m + 1}_{1}, \manifold{N})\)
such that \(\tr_{\Sset^m} U = u\) and for every \(t \in \intvo{0}{\infty}\),
\begin{multline}
\label{eq_Ufau0geiGh1Eang5voozehai}
 t^{m + 1}\mathcal{L}^{m + 1} \brk{\set{ x \in \Bset^{m + 1}_{1} \st \abs{\Deriv U (x)}\ge t}}\\
 \le A \exp \brk[\Bigg]{\hspace{1em}B
  \smashoperator{
    \iint_{\substack{
      (x, y) \in \Sset^m \times \Sset^m\\
      d (u (x), u (y))\ge \delta
      }}}
 \frac{1}{\abs{x - y}^{2m}} \dif x \dif y
 }
 \smashoperator[r]{\iint_{\Sset^m \times \Sset^m}}
 \frac{d(u (x), u(y))^{m + 1}}{\abs{x - y}^{2m}} \dif x \dif y.
\end{multline}
Moreover, one can take \(U\in C(\Bset^{m + 1}_{1}\setminus S, \manifold{N})\), where the singular set \(S \subset \Bset^{m + 1}_{1}\) is a finite set whose cardinality is controlled by the right-hand side of \eqref{eq_Ufau0geiGh1Eang5voozehai}.
\end{theorem}

The gap potential of the double integral appearing in the exponential in \eqref{eq_Ufau0geiGh1Eang5voozehai} 
\begin{equation}
\label{eq_uto9aefeaBimaGhaix9EeVoo}
  \smashoperator{\iint_{\substack{
      (x, y) \in \Sset^m \times \Sset^m\\
      d (u (x), u (y))\ge \delta
      }}}
 \frac{1}{\abs{x - y}^{2m}} \dif x \dif y
\end{equation}
first appeared in estimates by Bourgain, Brezis and Mironescu on the topological degree of maps from a sphere to itself
\citelist{\cite{Bourgain_Brezis_Nguyen_2005_CRAS}*{Th.\ 1.1}\cite{Bourgain_Brezis_Mironescu_2005_CPAM}*{Open problem 2}\cite{Nguyen_2007}} (see also \cite{Nguyen_2014}) and in estimates on free homotopy decompositions of mappings \cite{VanSchaftingen_2020}, as well as in estimates on liftings \citelist{\cite{Nguyen_2008_CRAS}*{Th.~2}\cite{VanSchaftingen_SumLift}}; they characterize, in the limit \(\delta \to 0\), first-order Sobolev spaces \citelist{\cite{Nguyen_2006}\cite{Nguyen_2007_CRAS}\cite{Nguyen_2008}} and encompass a property stronger than VMO \cite{Brezis_Nguyen_2011}.

\Cref{theorem_ball} implies, in particular, the weak-type estimates of \cite{Petrache_VanSchaftingen_2017}.
The improvement of \cref{theorem_ball} is two-fold: the dependence of the weak-type Sobolev bound on the extension \(U\) is exponential in the Gagliardo energy of \(u\), which is much more reasonable than the double exponential in \cite{Petrache_VanSchaftingen_2017} and the nonlinear part of the estimate, that is the exponential, relies on a gap potential \eqref{eq_uto9aefeaBimaGhaix9EeVoo} instead of a full fractional Gagliardo energy
\begin{equation*}
  \smashoperator{\iint_{\Sset^m \times \Sset^m}}
 \frac{d (u(x), u (y))^p}{\abs{x - y}^{2m}} \dif x \dif y.
\end{equation*}
The latter controls the former by the immediate estimate 
\begin{equation*}
   \smashoperator{\iint_{\substack{
      (x, y) \in \Sset^m \times \Sset^m\\
      d (u (x), u (y))\ge \delta
      }}}
 \frac{1}{\abs{x - y}^{2m}} \dif x \dif y
 \le \frac{1}{\delta^p} \smashoperator{\iint_{\Sset^m \times \Sset^m}}
 \frac{d (u(x), u (y))^p}{\abs{x - y}^{2m}} \dif x \dif y.
\end{equation*}

We also have a counterpart of \cref{theorem_ball} for the extension on the half-space \(\Rset^{m + 1 }_+ \defeq \Rset^m \times \intvo{0}{\infty}\) of mappings defined on the hyperplane \(\Rset^m \simeq \Rset^m \times \set{0}\).

\begin{theorem}
\label{theorem_halfspace}
Let \(m \in \Nset \setminus \set{0}\) and let \(\manifold{N}\) be a compact Riemannian manifold.
There exist constants \(A, B, \delta \in \intvo{0}{\infty}\) such that 
for every \(u \in W^{m/(m + 1), m + 1} (\Rset^m, \manifold{N})\), there
exists \(U \in W^{1, 1}_{\mathrm{loc}} (\Bar{\Rset}^{m + 1}_+, \manifold{N})\)
such that \(\tr_{\Rset^m} U = u\) and for every \(t \in \intvo{0}{\infty}\),
\begin{multline}
\label{eq_aiD8Oob2ahdah8lae0einei6}
 t^{m + 1} \mathcal{L}^{m + 1} \brk{\set{ x \in \Rset^{m + 1}_+ \st \abs{\Deriv U (x)}\ge t}}\\
 \le A \exp \brk[\Bigg]{\hspace{1em}B
  \smashoperator{
    \iint_{\substack{
      (x, y) \in \Rset^m \times \Rset^m\\
      d (u (x), u (y))\ge \delta
      }}}
 \frac{1}{\abs{x - y}^{2m}} \dif x \dif y
 }
 \smashoperator[r]{\iint_{\Rset^m \times \Rset^m}}
 \frac{d(u (x), u(y))^{m + 1}}{\abs{x - y}^{2m}} \dif x \dif y.
\end{multline}
Moreover, one can take \(U\in C(\Rset^{m + 1}_+ \setminus S, \manifold{N})\), where the singular set \(S \subset \Rset^{m + 1}_+\) is a finite set whose cardinality is controlled by the right-hand side of \eqref{eq_aiD8Oob2ahdah8lae0einei6}.
\end{theorem} 

Here \(U \in W^{1, 1}_{\mathrm{loc}} (\Bar{\Rset}^{m + 1}_+, \manifold{N})\) means that \(U\) is weakly differentiable and that for every compact set \(K \subset \Bar{\Rset}^{m + 1}_+\),
\(
 \int_{K} \abs{\Deriv U} < \infty
\).
It follows from \eqref{eq_aiD8Oob2ahdah8lae0einei6} that for every relatively finite-measure open set \(G \subset \Bar{\Rset}^{m + 1}_+\) and every \(q \in \intvr{1}{m + 1}\), we have
\(U \in W^{1, q} \brk{G, \manifold{N}}\).

Finally, we consider the extension to the hyperbolic space \(\Hset^{m + 1}\) of maps defined on its boundary sphere \(\Sset^m\). 
Using the Poincaré ball model, the hyperbolic space \(\Hset^{m + 1}\) is the ball \(\Bset^{m+1}_{1}\), endowed with the metric \(g_{\mathrm{hyp}}\) defined for \(x \in \Bset^{m + 1}_{1}\) in terms of the Euclidean metric \(g_{\mathrm{euc}}\) by 
\[
 g_{\mathrm{hyp}} (x) = \frac{4 g_{\mathrm{euc}}(x)}{\brk{1 - \abs{x}^2}^2},
\]
whose boundary \(\Sset^{m}\) is then considered to be the boundary sphere of \(\Hset^{m + 1}\).
In this setting, we have a hyperbolic counterpart of \cref{theorem_ball} and \cref{theorem_halfspace}.

\begin{theorem}
\label{theorem_hyperbolic}
Let \(m \in \Nset \setminus \set{0}\) and let \(\manifold{N}\) be a compact Riemannian manifold.
There exist constants \(A, B, \delta \in \intvo{0}{\infty}\) such that 
for every \(u \in W^{m/(m + 1), m + 1} (\Sset^m, \manifold{N})\), there
exists \(U \in W^{1, 1} (\Bset^{m + 1}_{1}, \manifold{N})\)
such that \(\tr_{\Sset^m} U = u\) and for every \(t \in \intvo{0}{\infty}\),
\begin{multline}
\label{eq_oogo0che0aevaeNgahshethu}
 \mathcal{H}^{m + 1} \brk{\set{ x \in \Hset^{m + 1} \st \abs{\Deriv U (x)}\ge t}}\\
 \le \frac{A}{t^{m + 1}} \exp \brk[\Bigg]{\hspace{1em}B
  \smashoperator{
    \iint_{\substack{
      (x, y) \in \Sset^m \times \Sset^m\\
      d (u (x), u (y))\ge \delta
      }}}
 \frac{1}{\abs{x - y}^{2m}} \dif x \dif y
 }
 \smashoperator[r]{\iint_{\Sset^m \times \Sset^m}}
 \frac{d(u (x), u(y))^{m + 1}}{\abs{x - y}^{2m}} \dif x \dif y.
\end{multline}
Moreover, one can take \(U\in C(\Hset^{m + 1}\setminus S, \manifold{N})\), where the singular set \(S \subset \Hset^{m + 1}\) is a finite set whose cardinality is controlled by the right-hand side of \eqref{eq_oogo0che0aevaeNgahshethu}.
\end{theorem}

Note that we assert that \(U \in W^{1, 1} (\Bset^{m + 1}_{1}, \manifold{N})\), with the ball \(\Bset^{m + 1}_{1}\) endowed with the Euclidean metric instead of the same ball \(\Hset^{m + 1}\) endowed with the hyperbolic metric. 
Indeed, \(U \in  W^{1, 1} (\Bset^{m + 1}_{1}, \manifold{N})\) would translate as 
\begin{equation*}
   \int_{\Hset^{m + 1}} \abs{\Deriv U} =  \int_{\Bset^{m + 1}_1} \frac{2^m \abs{\Deriv U \brk{x}}}{\brk{1 - \abs{x}^2}^m} \dif x < \infty,
\end{equation*}
which is neither a consequence of \(U \in  W^{1, 1} (\Bset^{m + 1}_{1}, \manifold{N})\) nor \(U \in  W^{1, 1} (\Hset^{m + 1}, \manifold{N})\) and has no reason to be expected.

Although standard conformal transformations between \(\Bset^{m+ 1}_{1}\), \(\Hset^{m + 1}\) and \(\Rset^{m + 1}_{+}\) preserve the fractional Gagliardo energy of the boundary values in the right-hand sides of \eqref{eq_Ufau0geiGh1Eang5voozehai}, \eqref{eq_aiD8Oob2ahdah8lae0einei6} and \eqref{eq_oogo0che0aevaeNgahshethu}, the corresponding gap potentials and the strong-type quantities
\begin{align*}
 &\int_{\Bset^{m + 1}_{1}}
 \abs{\Deriv U}^{m + 1},&
 &\int_{\Rset^{m + 1}_{+}}
 \abs{\Deriv U}^{m + 1}&
 &\text{ and }&
 &\int_{\Hset^{m + 1}}
 \abs{\Deriv U}^{m + 1}
\end{align*}
corresponding to the weak-type quantities in their left-hand side,
the left-hand sides as such are not conformally invariant so that \cref{theorem_ball,theorem_halfspace,theorem_hyperbolic} are not equivalent to each other.
In practice, this explains why we have the same construction of \(U\) in all three cases, but three different particular estimates on \(U\).

At the core of the proof of  \cref{theorem_ball,theorem_halfspace,theorem_hyperbolic} is a refined understanding of functions in the fractional critical Sobolev space \(W^{s, p} (\Rset^m, \Rset^\nu)\).
We start from the elegant and versatile approach through the vanishing mean oscillation property of critical Sobolev maps \cite{Brezis_Nirenberg_1995}: 
if \(sp = m\), it follows from Lebesgue's dominated convergence theorem and from the definition in \eqref{def of fractsobspace}
that 
\begin{equation}
\label{eq_uMoo2ohthae6ahs9Aukahw5e}
\lim_{r \to 0} \sup_{x \in \Rset^m} 
 \fint_{\Bset^{m}_r (x)} \fint_{\Bset^{m}_r (x)} \abs{u (y) - u (z)}^p \dif y \dif z 
  = 0.
\end{equation}
If \(u \in \manifold{N}\) almost everywhere, then we have 
\begin{equation}
\label{eq_bu1aiv1weet5Feohij5reeth}
\begin{split}
 \dist\brk[\Big]{\fint_{\Bset^{m}_r (x)} u, \manifold{N}}^{p}
 \le \fint_{\Bset^{m}_r (x)} \fint_{\Bset^{m}_r (x)} \abs{u (y) - u (z)}^p \dif y \dif z.
\end{split}
\end{equation}
Thus, by \eqref{eq_uMoo2ohthae6ahs9Aukahw5e} and \eqref{eq_bu1aiv1weet5Feohij5reeth},
\begin{equation}
\label{eq_Xoomeip5AhcheiLe2iSuoc7p}
   \lim_{r \to 0} \sup_{x \in \Rset^m} \dist \brk[\Big]{\fint_{\Bset^{m}_r (x)} u, \manifold{N}}^{p} = 0.
\end{equation}
Because of the lack of control on the rate of convergence of the limit in \eqref{eq_Xoomeip5AhcheiLe2iSuoc7p}, we need to refine it to get the quantitative results we are looking for.

To give a glimpse of our core quantitative estimate in a simplified setting,
let us  replace \eqref{eq_Xoomeip5AhcheiLe2iSuoc7p} by the integral inequality
\begin{equation}
\label{eq_jaiMeiTeiv8eicoopaiw2ziz}
\begin{split}
 \int_0^\infty  \sup_{x \in \Rset^m}& \fint_{\Bset^{m}_r (x)} \fint_{\Bset^{m}_r (x)} |u (y) - u (z)|^p \dif y \dif z \frac{\dif r}{r} \\
 &\le \frac{1}{\mathcal{L}^m (\Bset^{m}_1)^2}
 \int_0^\infty  
  \smashoperator[r]{\iint_{\substack{(y, z) \in \Rset^m \times \Rset^m\\ \abs{y - z} \le 2r}}} \abs{u (y) - u (z)}^p \dif y \dif z \frac{\dif r}{r^{2m+1}}\\
  &= \frac{4^{m}}{2m \mathcal{L}^m (\Bset^{m}_1)^2} \smashoperator[r]{\iint_{\Rset^m \times \Rset^m}} \frac{\abs{u (y) - u (z)}^p}{\abs{y - z}^{2m}} \dif y \dif z.
 \end{split}
\end{equation}
It follows from \eqref{eq_jaiMeiTeiv8eicoopaiw2ziz} and the Chebyshev inequality that for every \(\lambda \in \intvo{1}{\infty}\) there exists \(\tau \in \intvo{1}{\lambda}\) such that 
\begin{equation}
\label{eq_eowohY4vebuucooxooju1ohm}
\begin{split}
 \sum_{k \in \Zset} \sup_{x \in \Rset^m}&
 \dist \brk[\bigg]{\fint_{\Bset^{m}_{\tau \lambda^{-k}} (x)} u, \manifold{N}}^p \\
 &\le \sum_{k \in \Zset} \sup_{x \in \Rset^m}
 \fint_{\Bset^{m}_{\tau \lambda^{-k}} (x)} \fint_{\Bset^{m}_{\tau \lambda^{-k}} (x)} \abs{u (y) - u (z)}^p \dif y \dif z \\
 &\le \frac{4^m}{2m  \mathcal{L}^m (\Bset^{m}_1)^2 \ln \lambda} \smashoperator[r]{\iint_{\Rset^m \times \Rset^m}} \frac{\abs{u (y) - u (z)}^p}{\abs{y - z}^{2m}} \dif y \dif z.
\end{split}
\end{equation}
If we take now an extension \(u\) by convolution on the half-space \(\Rset^{m + 1}_+\), 
\eqref{eq_eowohY4vebuucooxooju1ohm} implies that when its right-hand side is small enough, that is, when 
\begin{equation}
\label{eq_hoh3aehaic7yeiSoyie7mein}
\ln \lambda \simeq \smashoperator{\iint_{\Rset^m \times \Rset^m}} \frac{\abs{u (y) - u (z)}^p}{\abs{y - z}^{2m}} \dif y \dif z , 
\end{equation}
there is a family of hyperplanes whose distance to the boundary is in geometric progression of ratio \(\lambda\) on which the extension by convolution is close to the manifold.
In view of \eqref{eq_hoh3aehaic7yeiSoyie7mein}, the ratio \(\lambda\) is controlled by an exponential of the Gagliardo fractional energy.

In order to actually prove the results, we improve the estimate \eqref{eq_jaiMeiTeiv8eicoopaiw2ziz} in two directions.
First, instead of considering parallel hyperplanes, we consider a decomposition of cubes on the boundary of which we control the mean oscillation.
Second, we replace the mean oscillation by a truncated mean oscillation which can be controlled by the gap potential \eqref{eq_uto9aefeaBimaGhaix9EeVoo}.

In comparison, the proof in \cite{Petrache_VanSchaftingen_2017} of \eqref{eq_eeyahGie3shee1ahth5luth9} relies on the estimate
(see \cite{Petrache_VanSchaftingen_2017}*{Prop.\ 4.2})
\begin{multline*}
\int_{0}^R 
 \sup_{\substack{x \in \Bset^{m + 1}_{1}\\ \abs{x}= r}} \dist\brk[\bigg]{\frac{(1 - r^2)^m}{\mathcal{H}^{m}(\Sset^m)} \int_{\Sset^m} \frac{u (y)}{\abs{x - y}^{2 m}} \dif y, \manifold{N}} \frac{\dif r}{1 - r^2} \\
 \le \int_{0}^R \frac{(1 - \abs{r}^2)^{2 m - 1}}{\mathcal{H}^{m}(\Sset^m)^2}
 \sup_{\substack{x \in \Bset^{m + 1}_{1}\\ \abs{x}= r}} \smashoperator[r]{\iint_{\Sset^m \times \Sset^m}} \frac{\abs{u (y) - u (z)}}{\abs{x - y}^{2 m}\abs{x - z}^{2 m}} \dif y \dif z \dif r \\
 \le C \brk[\Big]{\ln \frac{2}{1-R}}^{1 - \frac{1}{p}} \brk[\bigg]{\smashoperator[r]{\iint_{\Sset^m \times \Sset^m}} \frac{\abs{u (y)- u (z)}^p}{\abs{y - z}^{2 m}} \dif y \dif z}^\frac{1}{p},
\end{multline*}
with \(R \in \intvo{0}{1}\).
This allows to find bad balls in the hyperbolic space in the Poincaré ball models outside of which
the hyperharmonic extension to \(\Bset^{m+1}_{1}\) is close to the target manifold and whose radius is controlled by the exponential of the Gagliardo fractional energy. One cannot perform directly a homogeneous extension on the bad balls because they could intersect; through a Besicovitch-type covering argument, the construction is performed on collections of disjoint balls, with a number of collections bounded exponentially; combined with the exponential of the radius appearing in the equivalence between critical Marcinkiewicz--Sobolev quasinorms, this explains the double exponential appearing in the final estimate. In the present work we avoid this pitfall by working directly with a decomposition into disjoint cubes.

\section{Mean oscillation on \texorpdfstring{\(\lambda\)}{λ}--adic skeletons}

\subsection{Truncated mean oscillation}
Our analysis will rely on truncated mean oscillations, which have already been used in estimates on homotopy decompositions \cite{VanSchaftingen_2020} and on liftings over noncompact Riemannian coverings \cite{VanSchaftingen_SumLift}.

\begin{definition}
Given \(u : \Rset^{m} \to \manifold{N}\), we define the function \(\meanosc_{\delta, p} u : \Rset^{m + 1}_+ \to \intvc{0}{\infty}\) for every \(x = (x', x_{m +1}) \in \Rset^{m + 1}_+\) by
\begin{equation}
\label{eq_oJeeneithove3ahceid9Odah}
  \meanosc_{\delta, p} u (x)
  \defeq \fint_{\Bset^{m}_{x_{m +1}} (x')} \fint_{\Bset^{m}_{x_{m +1}} (x')} \brk{d (u (y), u (z)) - \delta}_+^p \dif y \dif z.
\end{equation}
\end{definition}

If \(1 \le p_1 \le p_2\), one has by Jensen's inequality,
\begin{equation}
\label{eq_fi2wiYou2sauM6roov6maode}
\brk{\meanosc_{\delta, p_1} u (x)}^{\frac{1}{p_1}}
\le \brk{\meanosc_{\delta, p_2} u (x)}^{\frac{1}{p_2}},
\end{equation}
whereas if \(\delta_0, \delta_1 \in \intvr{0}{\infty}\), one has by the triangle inequality and by convexity, 
\begin{equation}
\label{eq_eer1ong6xaiyaeWei3uo8iem}
  \meanosc_{\delta_1, p} u (x)
  \le 2^{p - 1}\brk{\meanosc_{\delta_0, p} u (x) + \brk{\delta_0 - \delta_1}_+^p}.
\end{equation}

The truncated Gagliardo fractional energy can be written in terms of mean oscillations as 
\begin{equation*}
 \smashoperator{\iint_{\Rset^m \times \Rset^m}}
 \frac{\brk*{d (u (y), u (z)) - \delta}_+^p}{\abs{y - z}^{m + sp}} \dif y \dif z
 \simeq \int_{\Rset^{m + 1}_+} \meanosc_{\delta, p} u (x) \frac{\dif x}{x_{m + 1}^{1 + sp}}.
\end{equation*}

\begin{proposition}\label{prop estimates for distVN and on DV}
Given \(\varphi \in C^\infty (\Rset^m, \Rset)\) such that \(\int_{\Rset^m} \varphi = 1\) and \(\supp \varphi \subseteq \Bset^{m}_1\),
there exists a constant \(C \in \intvo{0}{\infty}\), depending only on \(m\) and \(\varphi\), such that for every \(u \in L^1_{\mathrm{loc}}
\brk{\Rset^{m}, \Rset^\nu}\) and every \(Y \subseteq \Rset^\nu\) satisfying 
\(u \in Y\) almost everywhere in \(\Rset^m\),
if \(V : \Rset^{m + 1}_+ \to \Rset^\nu\) is defined for every \(\brk{x', x_{m + 1}}\) as 
\begin{equation}
\label{eq_shee0Ceingooghe2weefei4A}
 V (x) \defeq \frac{1}{x_{m + 1}^m} \int_{\Rset^m} u (z)\, \varphi \brk*{\tfrac{x' - z}{x_{m + 1}}} \dif z
 = \int_{\Rset^m} u(x' - x_{m + 1} z)\, \varphi(z) \dif z,
\end{equation}
then for every \(x = (x', x_{m + 1}) \in \Rset^{m + 1}_+\), for every $\delta \in [0,\infty)$ and every \(p \in \intvr{1}{\infty}\), 
\begin{equation}
\label{eq_phaejateiGh6ceir1kohngis}
\dist (V (x), Y)^p
\le C^p (\meanosc_{\delta, p} u (x)+ \delta^p)
\end{equation}
and 
\begin{equation}
\label{eq_taerinie4aWee3Phahzefeib}
  \abs{\Deriv V (x)}^p
  \le \frac{C^p}{x_{m+1}^p} \brk{\meanosc_{\delta, p} u (x)+ \delta^p}.
\end{equation}
\end{proposition}
\begin{proof}
First, we have by \eqref{eq_shee0Ceingooghe2weefei4A} and \eqref{eq_oJeeneithove3ahceid9Odah},
\begin{equation}
\label{estimate_mean_osc_1i49i9943r94i9}
\begin{split}
  \dist \brk{V \brk{x}, Y}
  &\le \fint_{\Bset^m_{x_{m + 1}}(x')} \abs{V \brk{x} - u \brk{y}} \dif y \\
  &\le \norm{\varphi}_{L^\infty \brk{\Rset^m}} \mathcal{L}^m \brk{\Bset^m_1}
  \fint_{\Bset^m_{x_{m + 1}}(x')} \fint_{\Bset^m_{x_{m + 1}}(x')} \abs{u \brk{z} - u \brk{y}} \dif z \dif y\\
  &\le \norm{\varphi}_{L^\infty \brk{\Rset^m}} \mathcal{L}^m \brk{\Bset^{m}_{1}}\meanosc_{0, 1} u (x),
\end{split}
\end{equation}
and the first conclusion \eqref{eq_phaejateiGh6ceir1kohngis} 
follows from \eqref{estimate_mean_osc_1i49i9943r94i9}, \eqref{eq_fi2wiYou2sauM6roov6maode} and \eqref{eq_eer1ong6xaiyaeWei3uo8iem}.

Next, defining \(\varphi_1 \colon \Rset^m \to \mathrm{Lin} \brk{\Rset^{m + 1}, \Rset}\)
for \(z \in \Rset^{m}\) and \(h = \brk{h', h_{m + 1}}\) by
\begin{equation*}
 \varphi_1 \brk{z}[\brk{h', h_{m + 1}}] \defeq \Deriv \varphi \brk{z}[h'] - h_{m + 1} \brk{m \varphi\brk{z} + \Deriv \varphi \brk{z} [z]},
\end{equation*}
we have for every \(x = \brk{x', x_{m + 1}} \in \Rset^{m + 1}_+\),
\begin{equation*}
\begin{split}
 \Deriv V \brk{x}
 &= \frac{1}{x_{m + 1}^{m + 1}} \int_{\Rset^m} u (z) \, \varphi_1 \brk*{\tfrac{x' - z}{x_{m + 1}}} \dif z \\
 &= \frac{1}{x_{m + 1}^{m + 1}} \fint_{\Bset^m_{x_{m + 1}}(x')}\int_{\Rset^m} \brk{u (z) - u(y)} \, \varphi_1 \brk*{\tfrac{x' - z}{x_{m + 1}}} \dif z \dif y,
\end{split}
\end{equation*}
from which we get 
\begin{equation*}
\abs{\Deriv V \brk{x}} \le \norm{\varphi_1}_{L^\infty \brk{\Rset^m}} \mathcal{L}^m \brk{\Bset^m_1}
  \fint_{\Bset^m_{x_{m + 1}}(x')} \fint_{\Bset^m_{x_{m + 1}}(x')} \abs{u \brk{z} - u \brk{y}} \dif z \dif y,
\end{equation*}
and we conclude as previously.
\end{proof}

\subsection{\texorpdfstring{\(\lambda\)}{λ}--adic cubes and skeletons}
The domain of the extension by convolution \(V\) given by \eqref{eq_shee0Ceingooghe2weefei4A} is the half-space \(\Rset^{m + 1}_+\); we will subdivide the latter in \(\lambda\)--adic cubes on which we will perform appropriate constructions. 

Given \(\lambda \in \intvo{1}{\infty}\) and \(\tau \in [1, \lambda]\), we consider for every \(k \in \Zset\) the set of cubes of \(\Rset^{m}\)
\begin{equation*}
 \mathcal{Q}_{\lambda, \tau, k}
 \defeq \set{  \tau \lambda^{-k} (\intvc{0}{1}^{m} + j) \st j \in \Zset^{m} }
\end{equation*}
and the corresponding set of cubes of \(\Rset^{m + 1}_+\)
\begin{equation*}
\begin{split}
 \mathcal{Q}^+_{\lambda, \tau, k}
 \defeq& \set{Q \times \intvc{\tfrac{\tau \lambda^{-k}}{\lambda - 1}}{\tfrac{\tau \lambda^{-(k-1)}}{\lambda - 1}}\st Q \in \mathcal{Q}_{\lambda, \tau, k}}\\
 =& \set{  \tau \lambda^{-k} (\intvc{0}{1}^{m + 1} +  (j, (\lambda - 1)^{-1}))
 \st j \in \Zset^{m} }.
 \end{split}
\end{equation*}
In particular, we have the decompositions
\begin{align*}
\Rset^{m} &= \bigcup_{Q \in \mathcal{Q}_{\lambda, \tau, k} } Q&
&\text{and}&
 \Rset^{m + 1}_+ &= \bigcup_{k \in \mathbb{Z}} \bigcup_{Q \in \mathcal{Q}^+_{\lambda, \tau, k} } Q.
\end{align*}

Moreover, we consider the part of the boundaries of the cubes in \(\mathcal{Q}_{\lambda, \tau, k}^+\) that are parallel to the hyperplane \(\Rset^m \times \set{0}\)
\begin{equation}
\label{eq_eWiel4oa8looqu8yohph3thi}
\begin{split}
 \mathcal{Q}^\parallel_{\lambda, \tau, k}
 \defeq& \set{Q \times \set{\tfrac{\tau \lambda^{-k}}{\lambda - 1}}\st Q \in \mathcal{Q}_{\lambda, \tau, k}}\\
 =& \set{  \tau \lambda^{-k} (\intvc{0}{1}^{m} \times \set{0} +  (j, (\lambda - 1)^{-1}))
 \st j \in \Zset^{m} }
 \end{split}
\end{equation}
and those that are normal to the same hyperplane
\begin{equation}
\label{eq_aecheiPhai2phee8Iechaib5}
\begin{split}
 \mathcal{Q}^\perp_{\lambda, \tau, k}
 \defeq& \set{\partial Q \times \intvc{\tfrac{\tau \lambda^{-k}}{\lambda - 1}}{\tfrac{\tau \lambda^{-(k-1)}}{\lambda - 1}} \st Q \in \mathcal{Q}_{\lambda, \tau, k}}\\
 =& \set{  \tau \lambda^{-k} ((\partial \intvc{0}{1}^{m}) \times \intvc{0}{1} +  (j, (\lambda - 1)^{-1}))
 \st j \in \Zset^{m} }.
 \end{split}
\end{equation}
Given \(h \in \Rset^m\), we also define the corresponding translated sets
\begin{equation} \label{dcubesf}
\begin{split}
  \mathcal{Q}_{\lambda, \tau, k, h}
  &\defeq \set{Q + \tau \lambda^{-k} h \st Q \in \mathcal{Q}_{\lambda, \tau, k}},\\
  \mathcal{Q}^+_{\lambda, \tau, k, h}
  &\defeq \set{Q + \tau \lambda^{-k} (h, 0) \st Q \in \mathcal{Q}^+_{\lambda, \tau, k}},\\
  \mathcal{Q}^\parallel_{\lambda, \tau, k, h}
  &\defeq \set{\Sigma + \tau \lambda^{-k} (h, 0) \st \Sigma \in \mathcal{Q}^\parallel_{\lambda, \tau, k}},\\
    \mathcal{Q}^\perp_{\lambda, \tau, k, h}
  &\defeq \set{\Sigma + \tau \lambda^{-k} (h, 0) \st \Sigma \in \mathcal{Q}^\perp_{\lambda, \tau, k}}.
  \end{split}
  \end{equation}

\subsection{Longitudinal estimate}
We first have an estimate on the maximal mean oscillation on longitudinal faces of cubes of the \(\lambda\)--adic decomposition of \(\Rset^{m + 1}_+\).

\begin{proposition}
\label{proposition_longitudinal}
For every \(m \in \Nset \setminus \set{0}\), there exists a constant \(C = C(m) \in \intvo{0}{\infty}\) such that for every \(p \in \intvr{1}{\infty}\), for every measurable map \(u : \Rset^m \to \manifold{N}\), for every \(\lambda \in \intvr{2}{\infty}\) and for every \(\delta \in \intvr{0}{\infty}\), one has 
\begin{equation}
\label{eq_ohTae0erahK4sohhohraec8s}
  \int_1^\lambda \sum_{k \in \Zset} \int_{\intvc{0}{1}^{m}} \sum_{\Sigma \in \mathcal{Q}_{\lambda, \tau, k, h}^\parallel} \sup_{x \in \Sigma} 
   \meanosc_{\delta, p} u (x) \dif h \frac{\dif \tau}{\tau}
   \le C 
 \smashoperator{\iint_{\Rset^{m} \times \Rset^{m}}} \frac{\brk{d \brk{u (y), u (z)} - \delta}_+^p}{\abs{y - z}^{2 m}} \dif y \dif z.
\end{equation}

\end{proposition}

\begin{proof}%
\resetconstant
For any \(\Sigma \in \mathcal{Q}_{\lambda, \tau, k, h}^\parallel\), we can write, in view of \eqref{eq_eWiel4oa8looqu8yohph3thi} and \eqref{dcubesf}, 
\[\Sigma = Q \times \set{ \tfrac{\tau \lambda^{-k}}{\lambda - 1}},
\]
where \(Q \in \mathcal{Q}_{\lambda, \tau, k, h}\). If \(x \in \Sigma\) and \(y \in \Rset^{m}\) satisfy \(\abs{x' - y} \le \tfrac{\tau \lambda^{-k}}{\lambda - 1}\), then \(x' \in Q\) and we have immediately that 
\begin{equation}\label{eq28_section2.1}
 \dist(y, Q) \le \abs{x' - y} \le \tfrac{\tau \lambda^{-k}}{\lambda - 1}.
\end{equation}
Since \(Q\) is a cube of edge length \(\tau \lambda^{-k}\), according to \eqref{eq28_section2.1}, we have
\begin{equation*}
y \in (1 + \tfrac{2}{\lambda - 1}) Q = \tfrac{\lambda + 1}{\lambda - 1} Q
\end{equation*}
under the convention that \(\tfrac{\lambda + 1}{\lambda - 1} Q\) is the cube with the same center as \(Q\) dilated by a factor \(\frac{\lambda + 1}{\lambda - 1}\). Thus, \(y \in  3Q
\), since \(\lambda \ge 2\).
One has then, by monotonicity of the integral, for every \(x \in \Sigma\),
\begin{equation}
\label{eq_euXungei0adee9aeX3Iec6Ee}
\begin{split}
 \meanosc_{\delta, p} u (x)
 &= \frac{1}{\mathcal{L}^{m}(\Bset^{m}_{1})^{2}} \brk[\Big]{\frac{ \lambda^k (\lambda - 1)}{\tau}}^{2m} \smashoperator{\iint_{\substack{(y, z) \in \Rset^{m} \times \Rset^{m}\\ \abs{y - x'}\le \frac{\tau \lambda^{-k}}{\lambda - 1}\\ 
 \abs{z - x'}\le \frac{\tau \lambda^{-k}}{\lambda - 1}
 }}} \brk{d \brk{u (y), u (z)} - \delta}_+^p \dif y \dif z\\
 &\le \frac{1}{\mathcal{L}^{m}(\Bset^{m}_{1})^{2}}\brk[\Big]{\frac{\lambda^k (\lambda - 1)}{\tau}}^{2m}\smashoperator{\iint_{\substack{(y,z)\in 3Q \times 3Q \\ \abs{y-z}\le \frac{2\tau\lambda^{-k}}{\lambda-1}}}} \brk{d \brk{u (y), u (z)} - \delta}_+^p \dif y \dif z.
 \end{split}
\end{equation} 
Summing \eqref{eq_euXungei0adee9aeX3Iec6Ee} over the sets \(\Sigma \in \mathcal{Q}_{\lambda, \tau, k, h}^\parallel\) and integrating the result over the translations \(h \in \intvc{0}{1}^m\), we get
\begin{multline}
\label{eq_quo2NiGhee4xooyeinahmi1v}
\int_{\intvc{0}{1}^{m}}
  \sum_{\Sigma \in \mathcal{Q}_{\lambda, \tau, k, h}^\parallel}   \sup_{x \in \Sigma} 
   \meanosc_{\delta, p} u (x) \dif h\\
   \le \frac{3^{m}}{\mathcal{L}^{m}(\Bset^{m}_{1})^{2}} \brk[\Big]{\frac{\lambda^k (\lambda - 1)}{\tau}}^{2m}\smashoperator{\iint_{\substack{(y, z) \in \Rset^{m} \times \Rset^{m}\\ \abs{y - z} \le \frac{2\tau \lambda^{-k}}{\lambda - 1}}}} \brk{d \brk{u (y), u (z)} - \delta}_+^p \dif y \dif z.
\end{multline}
Summing \eqref{eq_quo2NiGhee4xooyeinahmi1v} over the scales \(k \in \Zset\) and integrating the result over \(\tau \in \intvc{1}{\lambda}\), we get
\begin{equation*}
\label{eq_eeReuxah7thaighiich3Iich}
\begin{split}
   \int_1^\lambda \sum_{k \in \Zset} &\int_{\intvc{0}{1}^{m}}
  \sum_{\Sigma \in \mathcal{Q}_{\lambda, \tau, k, h}^\parallel}   \sup_{x \in \Sigma} 
   \meanosc_{\delta, p} u (x) \dif h
   \frac{\dif \tau}{\tau}\\
 &\le \frac{3^{m}(\lambda - 1)^{2 m}}{\mathcal{L}^{m}(\Bset^{m}_{1})^{2}}  
 \smashoperator[l]{\iint_{\Rset^{m} \times \Rset^{m}}} \sum_{k \in \Zset}
 \smashoperator[r]{\int_{\substack{\tau \in \intvo{1}{\lambda}\\ \tau \ge \frac{(\lambda - 1)\lambda^k \abs{y - z}}{2}}}}
 \frac{\lambda^{2k m } \brk{d \brk{u (y), u (z)} - \delta}_+^p}{\tau^{2 m + 1}} \dif \tau \dif y \dif z\\
 &= \frac{3^{m}(\lambda - 1)^{2 m}}{\mathcal{L}^{m}(\Bset^{m}_{1})^{2}}   
 \smashoperator[l]{\iint_{\Rset^{m} \times \Rset^{m}}} \sum_{k \in \Zset}
 \smashoperator[r]{\int_{\substack{\theta \in \intvo{\lambda^{-k}}{\lambda^{-\brk{k - 1}}} \\ \theta \ge \frac{(\lambda - 1)\abs{y - z}}{2}}}}
 \frac{\brk{d \brk{u (y), u (z)} - \delta}_+^p}{\theta^{2 m + 1}} \dif \theta \dif y \dif z\\
 &= \frac{3^{m}\brk{\lambda - 1}^{2m}}{\mathcal{L}^{m}(\Bset^{m}_{1})^{2}}  \smashoperator{\iint_{\Rset^{m} \times \Rset^{m}} }
 \int_{\frac{(\lambda - 1) \abs{y - z}}{2}}^\infty \frac{\brk{d \brk{u (y), u (z)} - \delta}_+^p}{\theta^{2 m + 1}} \dif \theta \dif y \dif z\\
 &= \frac{12^{m}}{2m\mathcal{L}^{m}(\Bset^{m}_{1})^{2}} 
 \smashoperator{\iint_{\Rset^{m} \times \Rset^{m}}} \frac{\brk{d \brk{u (y), u (z)} - \delta}_+^p}{\abs{y - z}^{2 m}} \dif y \dif z,
\end{split}
\end{equation*}
which implies the announced conclusion \eqref{eq_ohTae0erahK4sohhohraec8s} and completes our proof of \cref{proposition_longitudinal}.
\end{proof}

\subsection{Transversal estimate}
We next prove a counterpart of \cref{proposition_longitudinal}, where we estimate the maximal mean oscillation on transversal faces of cubes of the \(\lambda\)--adic decomposition of \(\Rset^{m + 1}_+\) instead of the longitudinal ones.

\begin{proposition}
\label{proposition_transversal}
For every \(m \in \Nset \setminus \set{0}\) and \(p \in \intvr{1}{\infty}\), there exists a constant \(C = C(m, p) \in \intvo{0}{\infty}\) such that for every measurable map \(u : \Rset^m \to \manifold{N}\), for every \(\lambda \in \intvr{2}{\infty}\) and for every \(\delta \in \intvr{0}{\infty}\), one has 
 \begin{equation}
 \label{eq_Quo8ief9gaidahfa0aebeir4}
   \int_1^\lambda \sum_{k \in \Zset} \int_{\intvc{0}{1}^{m}}
  \sum_{\Sigma \in \mathcal{Q}^\perp_{\lambda, \tau, k, h}}   \sup_{x \in \Sigma} 
   \meanosc_{\delta, p} u (x)\dif h \frac{\dif \tau}{\tau}\\
 \le C \smashoperator{\iint_{\Rset^{m} \times \Rset^{m}}} \frac{\brk{d \brk{u (y), u (z)} - \frac{\delta}{2}}_+^p}{\abs{y - z}^{2 m}} \dif y \dif z.
\end{equation}
 
\end{proposition}

\begin{proof}
We consider a set \(\Sigma \in \mathcal{Q}^\perp_{\lambda, \tau, k, h}\), that we can write, in view of \eqref{eq_aecheiPhai2phee8Iechaib5} and \eqref{dcubesf}, as 
\begin{equation}\label{001sigmadef001}
\Sigma = \partial Q \times \intvc{\tfrac{\tau \lambda^{-k}}{\lambda - 1}}{\tfrac{\tau \lambda^{-(k-1)}}{\lambda - 1}},
\end{equation}
where \(Q \in \mathcal{Q}_{\lambda, \tau, k, h}\).
We first note that by \eqref{eq_oJeeneithove3ahceid9Odah}, convexity and the triangle inequality,
\begin{equation}
\label{eq_rfoerjfojrifjtrgnvjfknnfg}
\begin{split}
\meanosc_{\delta, p} u (x)
&= \fint_{\Bset^{m}_{x_{m +1}} (x')} \fint_{\Bset^m_{x_{m +1}} (x')} \brk{d (u (y), u (z)) - \delta}_+^p \dif y \dif z \\
&\le 2^{p - 1} \brk[\bigg]{\fint_{\Bset^m_{x_{m +1}} (x')} \fint_{\Bset^m_{x_{m +1}} (x')} \fint_{E_{x}} \brk{d (u (y), u (w)) - \tfrac{\delta}{2}}_+^p  \dif y \dif z \dif w\\
& \qquad\qquad  +   \fint_{\Bset^m_{x_{m +1}} (x')} \fint_{\Bset^m_{x_{m +1}} (x')} \fint_{E_{x}} \brk{d (u (z), u (w)) - \tfrac{\delta}{2}}_+^p \dif y \dif z \dif w}\\
&= 2^{p} \fint_{\Bset^{m}_{x_{m +1}} (x')} \fint_{E_{x}}  \brk{d (u (y), u (z)) - \tfrac{\delta}{2}}_+^p \dif y \dif z,
\end{split}
\end{equation}
where 
\[
E_{x}\defeq \set{z \in \partial Q \st \abs{z-x'} \le x_{m+1}}.
\]
For every \(x = \brk{x', x_{m + 1}} \in \Sigma\), since \(\lambda \ge 2\) and \(Q\) is a cube of edge length \(\tau \lambda^{-k}\), we have \( \tau \lambda^{-k}\ge x_{m+1}\) and thus
\begin{equation}
\label{eq_ushuX0Ael8ci9xuu8jeem0ah}
 \mathcal{H}^{m - 1} \brk{E_x}
 \ge \frac{x_{m + 1}^{m-1}}{2^{m - 1}} \mathcal{L}^{m - 1} \brk{\Bset^{m - 1}_{1}}.
\end{equation}
One has then for every \(x \in \Sigma\), by \eqref{001sigmadef001}, \eqref{eq_rfoerjfojrifjtrgnvjfknnfg}, \eqref{eq_ushuX0Ael8ci9xuu8jeem0ah} and by monotonicity of the integral,
\begin{equation}
\label{eq_ieHis6iuLaec3iedohmaigoh}
\begin{split}
 \meanosc_{\delta, p} u (x)
 &
 \le \frac{\C}{x_{m+1}^{2m - 1}}
\smashoperator{\iint_{\substack{(y,z) \in \Rset^{m} \times \partial Q \\ \abs{y-x'} \le x_{m+1} \\ \abs{z-x'} \le x_{m+1}}}} \brk{d \brk{u (y), u (z)} - \tfrac{\delta}{2}}_+^p \dif y \dif z\\
 &\le \Cl{cst_aif9Haethah7xa5oedeV7xie}
 \smashoperator{\iint_{\substack{(y, z) \in \Rset^{m}\times \partial Q\\ \abs{y - z} \le \frac{2\tau\lambda^{-(k-1)}}{\lambda - 1}}}}
 \frac{\brk{d \brk{u (y), u (z)} - \frac{\delta}{2}}_+^p}{\abs{y  - z}^{2 m - 1}} \dif y \dif z,
 \end{split}
\end{equation} 
where the constant \(\Cr{cst_aif9Haethah7xa5oedeV7xie} \in (0, \infty)\) depends only on $m$ and $p$. Summing \eqref{eq_ieHis6iuLaec3iedohmaigoh} over the sets \(\Sigma \in \mathcal{Q}^\perp_{\lambda, \tau, k,h}\) and integrating the result with respect to the translations \(h\) over \(\intvc{0}{1}^{m}\), we have
\begin{equation}
\label{eq_FooxoN7thahyair0jeoy3aej}
\begin{split}
&\int_{\intvc{0}{1}^{m}}
  \sum_{\Sigma \in \mathcal{Q}^\perp_{\lambda, \tau, k, h}}   \sup_{x \in \Sigma }
   \meanosc_{\delta, p} u (x) \dif h \\
   &\hspace{2em}\le \Cr{cst_aif9Haethah7xa5oedeV7xie} \int_{\intvc{0}{1}^{m}}
   \sum_{Q \in \mathcal{Q}_{\lambda, \tau, k, h}}
 \smashoperator[r]{\iint_{\substack{(y, z) \in \Rset^m \times \partial Q\\ \abs{y - z} \le \frac{2\tau\lambda^{-(k-1)}}{\lambda - 1}}}} 
 \frac{\brk{d \brk{u (y), u (z)} - \frac{\delta}{2}}_+^p}{\abs{y  - z}^{2 m - 1}} \dif y \dif z \\
   &\hspace{2em}= \frac{\Cl{cst_seuveehaiShaifae6eijejah}  \lambda^k}{\tau} 
   \smashoperator{\iint_{\substack{(y, z) \in \Rset^{m} \times \Rset^{m}\\ \abs{y - z} \le \frac{2\tau\lambda^{-(k-1)}}{\lambda - 1}}}} \frac{\brk{d \brk{u (y), u (z)} - \frac{\delta}{2}}_+^p}{\abs{y - z}^{2 m - 1}} \dif y \dif z,
\end{split}
\end{equation}
where the constant \(\Cr{cst_seuveehaiShaifae6eijejah} \in (0, \infty)\) depends only on \(m\) and \(p\). In \eqref{eq_FooxoN7thahyair0jeoy3aej} we have used the fact that for every \(f \colon \Rset^{m} \to \intvc{0}{\infty}\), in view of  Fubini's theorem and change of variables, it holds
\[
\begin{split}
  \int_{\intvc{0}{1}^m} \sum_{Q \in \mathcal{Q}_{\lambda, \tau, k, h}} 
  \brk[\Big]{\int_{\partial Q} f \brk{z} \dif z} \dif h 
  &= \brk{\tau \lambda^{-k}}^{m - 1} \int_{\intvc{0}{1}^m} \sum_{Q \in \mathcal{Q}_{\lambda, 1, 0, h}}
  \brk[\Big]{\int_{\partial Q} f \brk{\tau \lambda^{-k} z} \dif z} \dif h \\
  & = 2m \brk{\tau \lambda^{-k}}^{m - 1} \int_{\Rset^m} f  \brk{\tau \lambda^{-k} z} \dif z\\
  & = \frac{2m \lambda^k}{\tau} \int_{\Rset^m} f.
\end{split}
\]
Thus, summing and integrating  \eqref{eq_FooxoN7thahyair0jeoy3aej} over the scales, we get
\begin{equation*}
\allowdisplaybreaks
\begin{split}
   \int_1^\lambda \sum_{k \in \Zset} &\int_{\intvc{0}{1}^{m}}
  \sum_{\Sigma \in \mathcal{Q}^\perp_{\lambda, \tau, k,h}}   \sup_{x \in \Sigma} 
   \meanosc_{\delta, p} u (x)\dif h \frac{\dif \tau}{\tau}\\
 &\le \Cr{cst_seuveehaiShaifae6eijejah}
 \smashoperator[l]{\iint_{\Rset^{m} \times \Rset^{m}}} \sum_{k \in \Zset}
 \smashoperator[r]{\int_{\substack{\tau \in \intvo{1}{\lambda}\\ \tau \ge
 \frac{\lambda^k (\lambda - 1)\abs{y - z}}
 {2\lambda}
 }}}
 \frac{\lambda^{k } \brk{d \brk{u (y), u (z)} - \frac{\delta}{2}}_+^p}{\abs{y - z}^{2 m - 1} \tau^{2}} \dif \tau \dif y \dif z\\
 &= \Cr{cst_seuveehaiShaifae6eijejah}  
 \smashoperator[l]{\iint_{\Rset^{m} \times \Rset^{m}}} \sum_{k \in \Zset}
 \smashoperator[r]{\int_{\substack{\theta \in \intvo{\lambda^{-k}}{\lambda^{-\brk{k - 1}}}\\ \theta\ge
 \frac{(\lambda - 1)\abs{y - z}}{2 \lambda}}
 }}
 \frac{\brk{d \brk{u (y), u (z)} - \frac{\delta}{2}}_+^p}{\abs{y - z}^{2 m - 1} \theta^{2}} \dif \theta \dif y \dif z\\
&= \Cr{cst_seuveehaiShaifae6eijejah}  
 \smashoperator[l]{\iint_{\Rset^{m} \times \Rset^{m}}}
 \int_{\frac{(\lambda-1)\abs{y - z}}{2\lambda}}^\infty
 \frac{\brk{d \brk{u (y), u (z)} - \frac{\delta}{2}}_+^p}{\abs{y - z}^{2 m - 1} \theta^{2}} \dif \theta \dif y \dif z\\
 &= \frac{2 \lambda} {\lambda - 1}\Cr{cst_seuveehaiShaifae6eijejah} 
 \smashoperator{\iint_{\Rset^{m} \times \Rset^{m}}} \frac{\brk{d \brk{u (y), u (z)} - \frac{\delta}{2}}_+^p}{\abs{y - z}^{2 m}} \dif y \dif z,
\end{split}
\end{equation*}
so that the conclusion \eqref{eq_Quo8ief9gaidahfa0aebeir4} follows, since \(\lambda \ge 2\).
\end{proof}

\subsection{Combining the estimates}
We close this section by summarizing \cref{proposition_longitudinal,proposition_transversal} in the following statement.

\begin{proposition}
\label{proposition_combined}
For every \(m \in \Nset \setminus \set{0}\) and \(p \in \intvr{1}{\infty}\), there exists a constant \(C =C(m,p)\in (0,\infty)\) such that for every measurable map \(u : \Rset^m \to \manifold{N}\), for every \(\lambda \in \intvr{2}{\infty}\) and for every \(\delta \in \intvr{0}{\infty}\), one has 
 \begin{equation*}
   \int_1^\lambda \sum_{k \in \Zset} \int_{\intvc{0}{1}^{m}}
  \sum_{Q \in \mathcal{Q}^+_{\lambda, \tau, k, h}}   \sup_{x \in \partial Q} 
   \meanosc_{\delta, p} u (x)\dif h \frac{\dif \tau}{\tau}\\
\le C \smashoperator{\iint_{\Rset^{m} \times \Rset^{m}}} \frac{\brk{d \brk{u (y), u (z)} - \frac{\delta}{2}}_+^p}{\abs{y - z}^{2 m}} \dif y \dif z.
\end{equation*}
\end{proposition}
\begin{proof}
 This follows immediately from \cref{proposition_longitudinal,proposition_transversal}.
\end{proof}

\section{Proofs of the singular extension theorems}

\subsection{Oscillation and gradient estimate on the skeleton}

We first estimate the average number of cubes on which the extension \(V\) of \(u\) given by \eqref{eq_shee0Ceingooghe2weefei4A} is far away from the range of \(u\).

\begin{proposition}
\label{proposition_goodcubes_gapintegral}
Let \(m \in \Nset \setminus \set{0}\).
There exist constants \(\eta \in \intvo{0}{1}\) and \(C \in \intvo{0}{\infty}\) depending only on $m$ such that for every \(\delta \in \intvo{0}{\infty}\), for every \(\lambda \in \intvr{2}{\infty}\), for every measurable function \(u : \Rset^m \to \Rset^\nu\) and every set \(Y \subseteq \Rset^\nu\), if \(V\) is an extension by convolution given by \eqref{eq_shee0Ceingooghe2weefei4A} and if \(u \in Y \) almost everywhere in \(\Rset^m\), then 
 \begin{multline}
 \label{eq_Aihohz6jeik2eeBioJ4wohgh}
   \int_1^\lambda \sum_{k \in \Zset} \int_{\intvc{0}{1}^{m}}
  \# \set {Q \in \mathcal{Q}^+_{\lambda, \tau, k,h} \st \sup_{x \in \partial Q} 
   \dist (V (x), Y) \ge \delta} \dif h \frac{\dif \tau}{\tau}\\
 \le \frac{C}{\delta^{m+1}}  \smashoperator{\iint_{\Rset^{m} \times \Rset^{m}}} \frac{\brk{d \brk{u (y), u (z)} - \eta \delta}_+^{m + 1}}{\abs{y - z}^{2 m}} \dif y \dif z.
\end{multline}
\end{proposition}
\begin{proof}%
\resetconstant
By \eqref{eq_phaejateiGh6ceir1kohngis}, we have for every \(x \in \Rset^{m + 1}_+\),
\begin{equation}
\label{eq_lah7eifooSeilee7zaiV8boo}
\dist (V (x), Y)^{m + 1}
\le \Cl{cst_moukeiguegoaphaiNg8Siu6h} \brk{\meanosc_{\theta, m + 1} u (x)+ \theta ^{m + 1}}.
\end{equation}
It is worth noting that, according to \eqref{eq_phaejateiGh6ceir1kohngis}, the constant \(\Cr{cst_moukeiguegoaphaiNg8Siu6h}\in (0, \infty)\) depends only on \(m\) and the function \(\varphi\) in the definition of \(V\), since \(p=m+1\). Thus, fixing the function \(\varphi\) in the definition of \(V\) from now on, we can assume that the constant \(\Cr{cst_moukeiguegoaphaiNg8Siu6h}\) depends only on \(m\). Hence, taking 
\begin{equation*}
 \eta \defeq \frac{1}{\brk[\big]{(1+2^{m+1})\Cr{cst_moukeiguegoaphaiNg8Siu6h}}^\frac{1}{m + 1}},
\end{equation*}
if 
\begin{equation}\label{eq_eq_01747}
 \dist (V (x), Y) \ge \delta,
\end{equation}
the inequality
\eqref{eq_lah7eifooSeilee7zaiV8boo} with \(\theta = 2 \eta \delta\) implies that 
\begin{equation}\label{eq_eq_017471}
(\eta \delta)^{m + 1}
\le \meanosc_{2 \eta \delta, m+1}u (x).
\end{equation}

We get, by \eqref{eq_eq_01747}, \eqref{eq_eq_017471} and \cref{proposition_combined} applied with \(p=m+1\),
 \begin{equation*}
 \begin{split}
   \int_1^\lambda \sum_{k \in \Zset} & \int_{\intvc{0}{1}^{m}}
  \# \set {Q \in \mathcal{Q}^+_{\lambda, \tau, k,h} \st \sup_{x \in \partial Q} 
   \dist (V (x), Y) \ge \delta} \dif h \frac{\dif \tau}{\tau}\\
  &\le \int_1^\lambda \sum_{k \in \Zset} \int_{\intvc{0}{1}^{m}}
  \# \set {Q \in \mathcal{Q}^+_{\lambda, \tau, k, h} \st \sup_{x \in \partial Q} 
   \meanosc_{2\eta \delta, m + 1} u(x) \ge (\eta \delta)^{m + 1}} \dif h \frac{\dif \tau}{\tau}\\
   &\le \frac{1}{(\eta \delta)^{m+1}} \int_1^\lambda \sum_{k \in \Zset} \int_{\intvc{0}{1}^{m}} \sum_{Q \in \mathcal{Q}^+_{\lambda, \tau, k, h}} \sup_{x \in \partial Q} 
   \meanosc_{2 \eta \delta, m + 1} u(x) \dif h \frac{\dif \tau}{\tau}\\
  &\le \frac{\Cr{cst_ahxu7teihaeTha5naex4peiR}}{\delta^{m+1}} \smashoperator{\iint_{\Rset^{m} \times \Rset^{m}}} \frac{\brk{d \brk{u (y), u (z)} - \eta \delta}_+^{m + 1}}{\abs{y - z}^{2 m}} \dif y \dif z,
  \end{split}
\end{equation*}
where the constant $\Cl{cst_ahxu7teihaeTha5naex4peiR} \in (0, \infty)$ depends only on $m$. This proves the estimate \eqref{eq_Aihohz6jeik2eeBioJ4wohgh} and completes our proof of Proposition~\ref{proposition_goodcubes_gapintegral}.
\end{proof}

Next, we can prove an average uniform bound on the extension by convolution \(V\).

\begin{proposition}
\label{proposition_goodcubes_sobolev}
Let \(m \in \Nset \setminus \set{0}\).
There exists a constant \(C=C(m)\in (0, \infty)\) such that for every \(\lambda \in \intvr{2}{\infty}\), for every measurable function \(u : \Rset^m \to \Rset^\nu\) , if \(V\) is an extension by convolution given by \eqref{eq_shee0Ceingooghe2weefei4A}, then 
 \begin{equation*}
 \label{eq_oodohh0Eighiedath2oe4ong}
   \int_1^\lambda \sum_{k \in \Zset} \int_{\intvc{0}{1}^{m}}
  \sum_{Q \in \mathcal{Q}^+_{\lambda, \tau, k,h}}
  \sup_{x \in \partial Q} x_{m +1}^{m + 1} \abs{\Deriv V (x)}^{m + 1}  \dif h \frac{\dif \tau}{\tau}
 \le C \smashoperator{\iint_{\Rset^{m} \times \Rset^{m}}} \frac{d \brk{u (y), u (z)}^{m + 1}}{\abs{y - z}^{2 m}} \dif y \dif z.
\end{equation*}
\end{proposition}

\begin{proof}
\resetconstant
We proceed similarly as in the proof of \cref{proposition_goodcubes_gapintegral}, where we assume that the function \(\varphi\) in the definition of \(V\) is fixed. Since by \eqref{eq_taerinie4aWee3Phahzefeib} applied with \(p=m+1\),
\begin{equation*}
  x_{m +1}^{m + 1} \abs{\Deriv V(x)}^{m + 1}
  \le  \Cl{mo3948jut9j459gujt9ungi} \brk{\meanosc_{\theta, m + 1} u (x)+ \theta ^{m + 1}},
\end{equation*}
where we can assume that \(\Cr{mo3948jut9j459gujt9ungi} \in (0, \infty)\) depends only on \(m\), the proof of Proposition~\ref{proposition_goodcubes_sobolev} then follows from \cref{proposition_combined}. Namely,
\begin{equation*}
\begin{split}
\int_1^\lambda \sum_{k \in \Zset} \int_{\intvc{0}{1}^{m}}
 & \sum_{Q \in \mathcal{Q}^+_{\lambda, \tau, k,h}}
  \sup_{x \in \partial Q} x_{m +1}^{m + 1} \abs{\Deriv V (x)}^{m + 1}  \dif h \frac{\dif \tau}{\tau} \\
 & \le \Cr{mo3948jut9j459gujt9ungi}\int_1^\lambda \sum_{k \in \Zset} \int_{\intvc{0}{1}^{m}}
  \sum_{Q \in \mathcal{Q}^+_{\lambda, \tau, k, h}}   \sup_{x \in \partial Q} 
   \meanosc_{0, m+1} u (x)\dif h \frac{\dif \tau}{\tau} \\
 & \le \Cl{n4939ujf99fuhu9h} \smashoperator{\iint_{\Rset^{m} \times \Rset^{m}}} \frac{d \brk{u (y), u (z)}^{m+1}}{\abs{y - z}^{2 m}} \dif y \dif z,
\end{split}
\end{equation*}
where the constant \(\Cr{n4939ujf99fuhu9h} \in (0, \infty) \)  depends only on \(m\).
\end{proof}

\subsection{Sobolev and Sobolev--Marcinkiewicz extensions on cubes}
The first construction we will perform is a classical extension of the boundary data on cubes with small energy. To simplify the presentation, the result is stated on a ball, which is bi-Lipschitzly equivalent to a cube.

\begin{lemma}
\label{lemma_extension}
Let \(m \in \Nset \setminus \set{0}\).
There exists a constant \(C = C(m) \in \intvo{0}{\infty}\) such that for every \(w \in 
W^{1, m + 1} (\Sset^m_\rho, \Rset^\nu)\) there exists a function 
\(W \in W^{1, m + 1} (\Bset^{m + 1}_\rho, \Rset^\nu) \cap C (\Bset^{m + 1}_\rho, \Rset^\nu)\) such that \(\tr_{\Sset^m_\rho} W = w\),
\begin{gather}
  \label{eq_aikahmi6foomei4teisu6ahY}
\int_{\Bset^{m + 1}_\rho} \abs{\Deriv W}^{m + 1}
 \le C \rho \int_{\Sset^{m}_\rho} \abs{\Deriv 
 w}^{m + 1},
\\
\label{eq_xie9ohch4engee6Iekohf1Ie}
\int_{\Bset^{m + 1}_\rho} \int_{\Sset^m_\rho} \abs{W (x) - w (y)}^{m + 1} \dif x
\dif y \le C \rho^{2m + 2} \displaystyle \int_{\Sset^{m}_\rho} \abs{\Deriv w}^{m + 1}\\
\intertext{and for every \(x \in \Bset^{m + 1}_\rho\) and almost every \(y \in \Sset^{m}_{\rho},\)}
\label{eq_ooGhohviu9koohe7zoogiPha}
\abs{W (x) - w (y)}^{m + 1} \le 
C\rho \int_{\Sset^m_{\rho}} \abs{\Deriv w}^{m + 1}.
\end{gather}
\end{lemma}
\begin{proof}
\resetconstant
Let
\begin{align*}
 p \defeq \frac{(m + 1)^2}{m}.
\end{align*}
We have, by the fractional Sobolev--Morrey embedding,
\begin{equation}
\label{eq_onoxihajaite5ceequie1Sho}
 \smashoperator[r]{\iint_{\Sset^m_\rho \times \Sset^m_\rho}}
 \frac{\abs{w (y) - w (z)}^p}{\abs{y - z}^{m + p  - 1}} \dif y \dif z
 \le \C \brk[\bigg]{\int_{\Sset^m_{\rho}} \abs{\Deriv w}^{m + 1}}^{\frac{p}{m+ 1}}.
\end{equation}
Assuming without loss of generality that \(\int_{\Sset^m_{\rho}} w = 0\), Gagliardo's classical trace theory \cite{Gagliardo_1957}  (see also \citelist{\cite{Leoni_2023}*{Th.\ 9.4}\cite{diBenedetto_2016}*{Prop.\ 17.1}\cite{Function}*{Section~6.9}\cite{Mazya_2011}*{Th.~10.1.1.1}}) yields a function \(W \in W^{1, p} (\Bset^{m+1}_\rho, \Rset^\nu)\) such that \(\tr_{\Sset^m_\rho} W = w\) and 
\begin{equation}
\label{eq_zeeLophiob6sheiyi3ieXi8b}
 \int_{\Bset^{m+1}_\rho} \abs{\Deriv W}^{p} + \frac{\abs{W}^{p}}{\rho^p}
 \le \C \smashoperator{\iint_{\Sset^m_\rho \times \Sset^m_\rho}} \frac{\abs{w (y) - w (z)}^{p}}{\abs{y - z}^{m + p - 1}} \dif y \dif z.
\end{equation}
It follows from Hölder's inequality,  \eqref{eq_zeeLophiob6sheiyi3ieXi8b} and \eqref{eq_onoxihajaite5ceequie1Sho}, that 
\begin{equation}
\label{eq_Pheith2daichoh8ExuuF3fod}
 \int_{\Bset^{m+1}_\rho} \abs{\Deriv W}^{m + 1} + \frac{\abs{W}^{m + 1}}{\rho^{m + 1}}
 \le \C\rho \int_{\Sset^m_{\rho}} \abs{\Deriv w}^{m + 1},
\end{equation}
and \eqref{eq_aikahmi6foomei4teisu6ahY} follows then from \eqref{eq_Pheith2daichoh8ExuuF3fod}.

Since for every \(x \in \Bset^{m+1}_\rho\) and \(y \in \Sset^{m}_\rho\),
\begin{equation*}
 \abs{W (x) - w (y)}^{m+1}\le 2^m \brk[\big]{\abs{W (x)}^{m+1}+\abs{w (y)}^{m+1}},
\end{equation*}
we have 
\begin{equation}\label{elementarysobest}
 \int_{\Bset^{m + 1}_\rho} \int_{\Sset^m_\rho} \abs{W (x) - w (y)}^{m + 1} \dif x \dif y
 \le \C \rho^{m} \int_{\Bset^{m + 1}_\rho} \abs{W}^{m + 1} + \C \rho^{m + 1} \int_{\Sset^m_\rho} \abs{w}^{m + 1}.
\end{equation}
The estimate \eqref{eq_xie9ohch4engee6Iekohf1Ie} follows from the estimates \eqref{elementarysobest},  
 \eqref{eq_Pheith2daichoh8ExuuF3fod} and the Poincaré inequality on the sphere.

Finally, by the Morrey--Sobolev inequality, \eqref{eq_zeeLophiob6sheiyi3ieXi8b} and \eqref{eq_onoxihajaite5ceequie1Sho}, we have 
\begin{equation*}
\begin{split}
 \abs{W (x) - w (y)}
 &\le \C \abs{x - y}^{1 - \frac{m + 1}{p}}
 \brk[\bigg]{\int_{\Bset^{m + 1}_\rho} \abs{\Deriv W}^p}^\frac{1}{p}\\
& \le \C \rho^{1 - \frac{m}{m + 1}}
 \brk[\bigg]{\int_{\Sset^m_\rho} \abs{\Deriv w}^{m + 1}}^\frac{1}{m + 1},
 \end{split}
\end{equation*}
and \eqref{eq_ooGhohviu9koohe7zoogiPha} follows.
\end{proof}

When the oscillation is too large on the boundary of cubes, we will perform our construction of the controlled singular extension; the resulting map is quite wild but is sufficiently well controlled to provide an acceptable extension on those cubes.

\begin{lemma}
\label{lemma_homogeneous}
Let \(m \in \Nset \setminus \set{0}\). If \(w \in W^{1, m + 1} (\Sset^{m}_\rho, \Rset^\nu)\) and if we define \(W : \Bset^{m + 1}_\rho \to \Rset^\nu\) for each \(x \in \Bset^{m + 1}_\rho \setminus \set{0}\) by 
\begin{equation*}
 W (x) \defeq w (\tfrac{\rho}{\abs{x}} x),
\end{equation*}
then \(W \in W^{1, 1} (\Bset^{m+1}_\rho, \Rset^\nu)\) and for every \(t \in \intvo{0}{\infty}\),
\begin{equation}
\label{eq_ao7iuceghaiphahVie8Quoor}
\mathcal{L}^{m + 1} \brk[\big]{\set[\big]{x \in \Bset_\rho^{m + 1} \st \abs{\Deriv W (x)}> t }}
  \le \frac{\rho}{(m + 1) t^{m +1}}  \int_{\Sset_\rho^{m}} \abs{\Deriv w}^{m + 1}.
\end{equation}
\end{lemma}
\begin{proof}
It can be observed immediately that \(W \in W^{1, m + 1} (\Bset^{m+1}_\rho \setminus \Bset^{m+1}_\varepsilon, \Rset^\nu)\) for every \(\varepsilon \in \intvo{0}{\rho}\) and that for every \(x \in \Bset^{m + 1}_{\rho} \setminus \set{0}\), one has 
\begin{equation*}
 \abs{\Deriv W (x)} = \tfrac{\rho}{\abs{x}} \abs{\Deriv w (\tfrac{\rho}{\abs{x}} x)}.
\end{equation*}
Hence, using Fubini's theorem, we have 
\begin{equation*}
\begin{split}
 \mathcal{L}^{m + 1} \brk[\big]{\set[\big]{x \in \Bset_\rho^{m + 1} \st \abs{\Deriv W (x)}> t }}
 &= 
 \int_0^\rho  
 \mathcal{H}^{m} \brk[\big]{\set[\big]{x \in \Sset^m_r \st \abs{\Deriv w (\tfrac{\rho}{r} x)} \ge \tfrac{t r}{\rho}}} \dif r \\
 &= \int_0^\rho  
 \mathcal{H}^{m} \brk[\big]{\set[\big]{y \in \Sset^m_\rho \st \abs{\Deriv w (y)} \ge \tfrac{t r}{\rho}}} \brk[\big]{\tfrac{r}{\rho}}^{m}\dif r \\
  &= \frac{\rho}{t^{m +1}}\int_0^t
 \mathcal{H}^{m} \brk[\big]{\set[\big]{y \in \Sset^m_\rho \st \abs{\Deriv w (y)} \ge \tau}} \tau^m \dif \tau\\
 &\le \frac{\rho}{t^{m +1}}\int_0^\infty 
 \mathcal{H}^{m} \brk[\big]{\set[\big]{y \in \Sset^m_\rho \st \abs{\Deriv w (y)} \ge \tau}} \tau^m \dif \tau\\
 &=\frac{\rho}{t^{m+1}}\int_{0}^{\infty} \mathcal{H}^{m} \brk[\big]{\set[\big]{y \in \Sset^m_\rho \st \abs{\Deriv w (y)}^{m+1} \ge \tau}}  \frac{\dif \tau}{m+1}\\
 &= \frac{\rho}{(m + 1) t^{m +1}}  \int_{\Sset_\rho^{m}} \abs{\Deriv w(x)}^{m + 1} \dif x.
\end{split}
\end{equation*}
This yields \eqref{eq_ao7iuceghaiphahVie8Quoor} and completes our proof of Lemma~\ref{lemma_homogeneous}.
\end{proof}

\subsection{Proofs of the theorems}

We first construct and estimate the singular extension on the half-space (\cref{theorem_halfspace}).

\begin{proof}[Proof of \cref{theorem_halfspace}]
\resetconstant
We fix \(\delta_{\manifold{N}} \in \intvo{0}{\infty}\) so that the nearest-point retraction \(\Pi_{\manifold{N}} : \manifold{N} + \Bset^\nu_{\delta_{\manifold{N}}} \to \manifold{N}\) is well defined and smooth up to the boundary.
We take \(V : \Rset^{m + 1}_+ \to \Rset^\nu\) to be an extension by convolution of \(u\) to \(\Rset^{m + 1}_{+}\) as in \eqref{eq_shee0Ceingooghe2weefei4A}.

Since by assumption \(u \in \manifold{N}\) almost everywhere on \(\Rset^m\), by the averaged estimate on the distance to the target (see
\cref{proposition_goodcubes_gapintegral}), we have
\begin{multline}\label{eq_kfmrefijiu4ugntvri}
   \int_1^\lambda \sum_{k \in \Zset} \int_{\intvc{0}{1}^{m}}
  \# \set[\Big]{Q \in \mathcal{Q}^+_{\lambda, \tau, k, h} \st \sup_{x \in \partial Q} 
   \dist (V (x), \manifold{N}) \ge \delta_{\manifold{N}}/2} \dif h \frac{\dif \tau}{\tau}\\
 \le \Cl{cst_yiek0OoRe0hielie4pa5shee}  \smashoperator{\iint_{\substack{(y, z) \in \Rset^{m} \times \Rset^{m}\\ d (u(y), u(z)) \ge \eta\delta_{\mathcal{N}}/2}}} \frac{1}{\abs{y - z}^{2 m}} \dif y \dif z,
 \end{multline}
where \(\eta>0\) is the constant of Proposition~\ref{proposition_goodcubes_gapintegral} depending only on \(m\), the constant \(\Cr{cst_yiek0OoRe0hielie4pa5shee} \in \intvo{0}{\infty}\) depends only on $m$, $\mathcal{N}$, and we have also used that $\mathcal{N}$ is compact, namely $\diam(\mathcal{N})<\infty$. We set \(\delta\defeq \eta \delta_{\mathcal{N}}/2\). Taking
\begin{equation}
\label{eq_quoh4pooKo4wohtequ6heoju}
 \lambda \defeq 1 + \exp \brk[\Bigg]{2 \Cr{cst_yiek0OoRe0hielie4pa5shee}  \smashoperator{\iint_{\substack{(y, z) \in \Rset^{m} \times \Rset^{m}\\ d (u(y), u(z)) \ge \delta}}} \frac{1}{\abs{y - z}^{2 m}} \dif y \dif z}
\end{equation}
and using \eqref{eq_kfmrefijiu4ugntvri}, we have then 
\begin{equation}
\label{eq_pazaezuc4Aisielei9oowee6}
   \frac{1}{\ln \lambda} \int_1^\lambda \sum_{k \in \Zset} \int_{\intvc{0}{1}^{m}}
  \# \set[\Big]{Q \in \mathcal{Q}^+_{\lambda, \tau, k, h} \st \sup_{x \in \partial Q} 
   \dist (V (x), \manifold{N}) \ge \delta_{\mathcal{N}}/2} \dif h \frac{\dif \tau}{\tau}
 \le \frac{1}{2}.
 \end{equation}
Similarly, by \cref{proposition_goodcubes_sobolev},
 \begin{multline}
 \label{eq_Sei5aiph6ahva9Zie2ooyika}
   \frac{1}{\ln \lambda}\int_1^\lambda \sum_{k \in \Zset} \int_{\intvc{0}{1}^{m}}
  \sum_{Q \in \mathcal{Q}^+_{\lambda, \tau, k,h}}
  \sup_{x \in \partial Q} x_{m +1}^{m + 1} \abs{\Deriv V (x)}^{m + 1}  \dif h \frac{\dif \tau}{\tau}\\
 \le \frac{\C}{\ln \lambda} \smashoperator{\iint_{\Rset^{m} \times \Rset^{m}}} \frac{d \brk{u (y), u (z)}^{m + 1}}{\abs{y - z}^{2 m}} \dif y \dif z.
\end{multline} 
In view of the estimates \eqref{eq_pazaezuc4Aisielei9oowee6} and \eqref{eq_Sei5aiph6ahva9Zie2ooyika}, there exists \(\tau \in \intvo{1}{\lambda}\) and, for every \(k \in \Zset\), \(h_k \in \intvc{0}{1}^m\) such that for every \(k \in \Zset\), every \(Q \in \mathcal{Q}^+_{\lambda, \tau, k, h_k}\) and every \(x \in \partial Q\), one has 
\begin{equation}
\label{eq_sheeceeGh5eeyohLeineiwuC}
 \dist (V (x), \manifold{N}) \le \frac{\delta_{\manifold{N}}}{2}
\end{equation}
and
\begin{equation}
 \label{eq_eghoh1oe0nuchiev6lahPaey}
   \sum_{k \in \Zset} 
  \sum_{Q \in \mathcal{Q}^+_{\lambda, \tau, k, h_k}}
  \sup_{x \in \partial Q} x_{m + 1}^{m + 1} \abs{\Deriv V (x)}^{m + 1}
\le \frac{\Cl{esrnvo3i4j93juvirhn}}{\ln \lambda} \smashoperator[r]{\iint_{\Rset^{m} \times \Rset^{m}}} \frac{d \brk{u (y), u (z)}^{m + 1}}{\abs{y - z}^{2 m}} \dif y \dif z.
\end{equation} 

We define now the set of good cubes
\begin{equation}\label{eq45iimio_good_cubes}
 \mathcal{G} \defeq 
  \set[\Big]{Q \in \mathcal{Q}^+_{\lambda, \tau, k, h_k} \st 
 \sup_{x \in \partial Q} x_{m +1}^{m + 1} \abs{\Deriv V(x)}^{m + 1}
  \le \mu \text{ and } k \in \Zset}
\end{equation}
and the set of bad cubes
\begin{equation}\label{eqmi5jj5njnhj_notgood_cubes}
 \mathcal{B} \defeq 
  \set[\Big]{Q \in \mathcal{Q}^+_{\lambda, \tau, k, h_k} \st \sup_{x \in \partial Q} x_{m +1}^{m + 1} \abs{\Deriv V(x)}^{m + 1}
  > \mu \text{ and } k \in \Zset},
\end{equation}
where \(\mu\) will be chosen in \eqref{eq_ugii0ce5oa6haiGhieb8oohe}. Clearly, any cube is either good or bad and thus 
\begin{equation*}
 \bigcup_{Q \in \mathcal{G} \cup \mathcal{B}} Q = \Rset^{m + 1}_+.
\end{equation*}
By \eqref{eq_eghoh1oe0nuchiev6lahPaey} and \eqref{eqmi5jj5njnhj_notgood_cubes},
\begin{equation}\label{eq_irj398493j4gu}
\begin{split}
  \# \mathcal{B}
  &\le  \frac{1}{\mu} \sum_{k \in \Zset} 
  \sum_{Q \in \mathcal{Q}^+_{\lambda, \tau, k, h_k}}
  \sup_{x \in \partial Q} x_{m + 1}^{m + 1} \abs{\Deriv V (x)}^{m + 1}\\
&\le \frac{\Cr{esrnvo3i4j93juvirhn}}{\mu \ln \lambda} \smashoperator{\iint_{\Rset^{m} \times \Rset^{m}}}
\frac{d \brk{u (y), u (z)}^{m + 1}}{\abs{y - z}^{2 m}} \dif y \dif z.
\end{split}
\end{equation}
Notice  also that for every \(k \in \Zset\) and every \(Q \in \mathcal{Q}^+_{\lambda, \tau, k, h_k}\), we have
\begin{equation}
\label{eq_yoh6Ahl3aiba6jie3xoj3ael}
\begin{split}
\tau \lambda^{-k} \int_{\partial Q} \abs{\Deriv V}^{m + 1}
 &\le \Cl{cst_ieceishu2Ahleej9aizoh7EN} (\tau \lambda^{-k})^{m + 1} \sup_{x \in \partial Q } \abs{\Deriv V(x)}^{m + 1}\\
 &\le \Cr{cst_ieceishu2Ahleej9aizoh7EN} (\lambda - 1)^{m + 1} \sup_{x \in \partial Q} x_{m +1}^{m + 1} \abs{\Deriv V(x)}^{m + 1},
\end{split} 
\end{equation}
since for every \(x \in Q\), \(x_{m +1} \ge \tau \lambda^{-k}/(\lambda - 1)\), in view of \eqref{dcubesf}.

We are now going to define a map \(W : \Rset^{m+1}_+ \to \manifold{N} + \Bset^{\nu}_{\delta_{\manifold{N}}}\) separately on \(\bigcup_{\mathcal{G}}\) and \(\bigcup_{\mathcal{B}}\).

For every \(Q \in \mathcal{G}\), we apply \cref{lemma_extension}, up to a suitable bi-Lipschitz homeomorphism between a ball and a cube (see \cite{Griepentrog_Hoppner_Kaiser_Rehberg_2008}*{Cor.~3}), to define the mapping \(W\) on \(Q\) as an extension of \(V \restr{ \partial Q}\).
We have, in view of \eqref{eq_aikahmi6foomei4teisu6ahY} and of \eqref{eq_yoh6Ahl3aiba6jie3xoj3ael},
\begin{equation}
\label{eq_waeLee9quoh8aphaew8Iej1b}
\begin{split}
\int_{Q} \abs{\Deriv W}^{m + 1}
 &\le \C \tau \lambda^{-k} \int_{\partial Q} \abs{\Deriv V}^{m + 1}\\
 &\le  \C (\lambda - 1)^{m + 1} \sup_{x \in \partial Q} x_{m +1}^{m + 1} \abs{\Deriv V(x)}^{m + 1},
\end{split}
\end{equation}
whereas by the triangle inequality, \eqref{eq_sheeceeGh5eeyohLeineiwuC}, \eqref{eq_ooGhohviu9koohe7zoogiPha}, \eqref{eq_yoh6Ahl3aiba6jie3xoj3ael} and \eqref{eq45iimio_good_cubes},
\begin{equation*}
  \dist(W (x), \manifold{N})
  \le \frac{\delta_\manifold{N}}{2} + \Cl{cst_thowuofoaZiizi2Ahw6hoowe} \mu^{\frac{1}{m+1}} (\lambda - 1).
\end{equation*}
Hence if
\begin{equation}
\label{eq_ugii0ce5oa6haiGhieb8oohe}
 \mu =
 \brk[\Bigg]{\frac{\delta_{\mathcal{N}}}{2 \Cr{cst_thowuofoaZiizi2Ahw6hoowe} (\lambda - 1)}}^{m + 1},
\end{equation}
we have for every \(x \in Q\), \(W (x) \in \manifold{N} + \Bset^{\nu}_{\delta_{\manifold{N}}}\). Using the Chebyshev inequality and \eqref{eq_waeLee9quoh8aphaew8Iej1b},
we obtain
\begin{equation}\label{macheormtbnoi}
t^{m+1}\mathcal{L}^{m+1}\brk{\set{x \in Q \st \abs{\Deriv W (x)} \ge t}} \le \C (\lambda - 1)^{m + 1} \sup_{x \in \partial Q} x_{m +1}^{m + 1} \abs{\Deriv V(x)}^{m + 1}.
\end{equation}

Next, we apply, up to a suitable bi-Lipschitz homeomorphism between a ball and a cube (see \cite{Griepentrog_Hoppner_Kaiser_Rehberg_2008}*{Cor.~3}), \cref{lemma_homogeneous} on every bad cube \(Q \in \mathcal{B}\) to define there \(W\) by homogeneous extension of \(V \restr{\partial Q}\) with respect to the barycenter of \(Q\).
We compute then for such a cube, in view of \eqref{eq_ao7iuceghaiphahVie8Quoor} and \eqref{eq_yoh6Ahl3aiba6jie3xoj3ael}, for every \(t \in \intvo{0}{\infty}\),
\begin{equation}
\label{eq_cheijo5gee6ahphieQui1sho}
\begin{split}
 \mathcal{L}^{m + 1} \brk{\set{x \in Q \st \abs{\Deriv W (x)} \ge t}}
 &\le \frac{\C}{ t^{m + 1}} \tau \lambda^{-k} \int_{\partial Q} \abs{\Deriv V}^{m + 1}\\
 &\le \frac{\C}{ t^{m + 1}} (\lambda - 1)^{m + 1} \sup_{x \in \partial Q} x_{m +1}^{m + 1} \abs{\Deriv V(x)}^{m + 1}.
\end{split}
\end{equation}
Combining the estimates \eqref{macheormtbnoi} and \eqref{eq_cheijo5gee6ahphieQui1sho}, we get
\begin{equation}
\label{eq_ohG1eoMooghaiwoi7bei7iep}
\begin{split}
t^{m + 1}\mathcal{L}^{m + 1} &\brk{\set{x \in \mathbb{R}^{m+1}_{+} \st \abs{\Deriv W (x)} \ge t}}\\
 &\le \C (\lambda - 1)^{m + 1} \sum_{k \in \Zset} \sum_{Q \in \mathcal{Q}^+_{\lambda, \tau, k, h_k}} \sup_{x \in \partial Q} x_{m +1}^{m + 1} \abs{\Deriv V(x)}^{m + 1}\\
 &\le \Cl{3k4ri934jf9j5guj4ugh} \exp \brk[\Bigg]{\hspace{1em}\Cl{j3fon49ijf9jgfu95guinh}
  \smashoperator{
    \iint_{\substack{
      (y, z) \in \Rset^m \times \Rset^m\\
      d (u (y), u (z))\ge \delta
      }}}
 \frac{1}{\abs{y - z}^{2m}} \dif y \dif z
 }
 \smashoperator[r]{\iint_{\Rset^m \times \Rset^m}}
 \frac{d(u (y), u(z))^{m + 1}}{\abs{y - z}^{2m}} \dif y \dif z,
 \end{split}
\end{equation}
in view of \eqref{eq_quoh4pooKo4wohtequ6heoju} and \eqref{eq_eghoh1oe0nuchiev6lahPaey}, where $\Cr{3k4ri934jf9j5guj4ugh}$ and $\Cr{j3fon49ijf9jgfu95guinh}$ are positive constants depending only on $m$, $\mathcal{N}$.
Moreover, if \(S\) denotes the set of the barycenters of the cubes \(Q \in \mathcal{B}\), we have 
\begin{equation}
\label{eq_Aepaikoh6Lei6Hoh8pabuloe}
\# S = \# \mathcal{B}
\le  \Cl{m94j5g9hutgy4uihg} \exp \brk[\Bigg]{\hspace{1em} \Cl{n239hru4hfy84h5gyi5u4gnuin}
  \smashoperator{
    \iint_{\substack{
      (y, z) \in \Rset^m \times \Rset^m\\
      d (u (y), u (z))\ge \delta
      }}}
 \frac{1}{\abs{y - z}^{2m}} \dif y \dif z
 }
 \smashoperator[r]{\iint_{\Rset^m \times \Rset^m}}
 \frac{d(u (y), u(z))^{m + 1}}{\abs{y - z}^{2m}} \dif y \dif z,
\end{equation}
in view of \eqref{eq_quoh4pooKo4wohtequ6heoju}, \eqref{eq_irj398493j4gu} and \eqref{eq_ugii0ce5oa6haiGhieb8oohe}, where $\Cr{m94j5g9hutgy4uihg}$ and $\Cr{n239hru4hfy84h5gyi5u4gnuin}$ are positive constants depending only on $m$, $\mathcal{N}$.

In order to check that \(u\) is the trace of \(W\), we observe that on the one hand we have, since \(V\) is an extension by convolution, for every \(Q \in \mathcal{G}\), by Poincaré's inequality, by \eqref{eq_xie9ohch4engee6Iekohf1Ie} and by \eqref{eq_yoh6Ahl3aiba6jie3xoj3ael},
\begin{equation}
\label{eq_ac6thaighahxahGh1mahb4ai}
\begin{split}
 \int_{Q}& \abs{V - W}^{m + 1}\\
 &\le \frac{\Cl{ni4rj3j4uhtu9h435n}}{\brk{\tau \lambda^{-k}}^m}
 \brk*{\int_{\partial Q} \int_{Q} \abs{V(x) - V (y)}^{m + 1 } \dif y \dif x
 + \int_{\partial Q} \int_{Q} \abs{W (x) - V (y)}^{m + 1 } \dif y \dif x
 }\\
 &\le \frac{\Cr{ni4rj3j4uhtu9h435n}}{(\tau \lambda^{-k})^m} \int_{Q} \int_{\partial Q} \abs{V(x) - V (y)}^{m + 1 } \dif x \dif y
 + \C (\tau \lambda^{-k})^{m + 2} \int_{\partial Q} \abs{\Deriv V}^{m + 1 }\\
 &\le \C(\tau \lambda^{-k})^{m + 1} \brk*{\int_{Q} \abs{\Deriv V}^{m + 1}
 +
  (\lambda - 1)^{m + 1} \sup_{x \in \partial Q} x_{m + 1}^{m + 1} \abs{\Deriv V (x)}^{m + 1}}.
\end{split}
\end{equation}
By the classical theory of traces, 
\begin{equation}
\label{eq_lahXef3uyieheeng3ohwahH0}
 \int_{\Rset^{m + 1}} \abs{\Deriv V}^{m + 1} <\infty.
\end{equation}
It follows then from \eqref{eq_eghoh1oe0nuchiev6lahPaey}, \eqref{eq_ac6thaighahxahGh1mahb4ai} and \eqref{eq_lahXef3uyieheeng3ohwahH0} that 
\begin{equation}\label{eq_knreinebir9043rj4iuhu}
\sum_{k \in \mathbb{Z}} \sum_{Q \in \mathcal{G}} \frac{1}{(\lambda \tau^{-k})^{m + 1}}
 \int_{Q} \abs{V - W }^{m + 1} < \infty.
\end{equation}
Thus, since the set \(\mathcal{B}\) is finite, in view of \eqref{eq_knreinebir9043rj4iuhu}, we have 
\begin{equation*}
 \lim_{k \to \infty}
 \frac{1}{(\lambda \tau^{-k})^{m + 1}}
 \sum_{Q \in \mathcal{Q}^{+}_{\lambda, \tau, k, h_k}} \int_{Q} \abs{V - W }^{m + 1}\dif x = 0.
\end{equation*}
This implies that \(W\) and \(V\) have the same trace, and hence \(u\) is the trace of \(W\).

Finally, we define  \(U \defeq \Pi_{\manifold{N}} \compose W\).
The map \(U\) also has \(u\) as the trace on \(\Rset^m \times \set{0}\); the conclusion \eqref{eq_aiD8Oob2ahdah8lae0einei6} follows from \eqref{eq_ohG1eoMooghaiwoi7bei7iep};
\(U\in C (\Rset^{m + 1}_{+} \setminus S, \manifold{N})\) with the cardinality of the singular set \(S\) being estimated by \eqref{eq_Aepaikoh6Lei6Hoh8pabuloe}. This completes our proof of Theorem~\ref{theorem_halfspace}.
\end{proof}

The proof of \cref{theorem_halfspace} can be adapted easily to the case of the hyperbolic space (\cref{theorem_hyperbolic}).

\begin{proof}%
[Proof of \cref{theorem_hyperbolic}]%
\resetconstant
We proceed as in the proof of \cref{theorem_halfspace}.
Instead of \eqref{eq_cheijo5gee6ahphieQui1sho}, we proceed using the Poincaré metric on the half-space,  \eqref{eq_ao7iuceghaiphahVie8Quoor},  \eqref{eq_yoh6Ahl3aiba6jie3xoj3ael} and get for every \(Q \in \mathcal{Q}^{+}_{\lambda, \tau, k, h}\) (see \eqref{dcubesf}),
\begin{equation*}
\label{eq_UunifeeleaKaeg8ui6eezuej}
\begin{split}
 \smashoperator[r]{\int_{\substack{x \in Q\\ \abs{\Deriv W (x)}x_{m +1} \ge t}}}
 \frac{1}{x_{m +1}^{m + 1}} \dif x
 &\le \brk[\Big]{\frac{(\lambda - 1)\lambda^k}{\tau}}^{m + 1}
 \mathcal{L}^{m + 1}\brk{\set{x \in Q \st \abs{\Deriv W (x)}\tfrac{\tau \lambda^{-(k - 1)}}{\lambda - 1} \ge t}}\\[-1em]
 &\le \frac{\C \lambda^{m + 1}}{t^{m + 1}}  \tau \lambda^{-k} \int_{\partial Q} |\Deriv V|^{m+1}\\
 &\le \frac{\Cl{goodconst21234567} (\lambda-1)^{m + 1}}{t^{m + 1}}
 (\tau \lambda^{-k})^{m+1}  \sup_{x \in \partial Q} \abs{\Deriv V (x)}^{m + 1}\\
 &\le  \frac{\Cr{goodconst21234567} (\lambda-1)^{2m + 2}}{t^{m + 1}}  \sup_{x \in \partial Q} x^{m+1}_{m+1} \abs{\Deriv V (x)}^{m + 1},
\end{split}
\end{equation*}
where \(\lambda \geq 2\), \(V\) is defined in \eqref{eq_shee0Ceingooghe2weefei4A} and, up to a suitable bi-Lipschitz homeomorphism between a ball and a cube (see \cite{Griepentrog_Hoppner_Kaiser_Rehberg_2008}*{Cor.~3}), \(W\) is defined on \(Q\) by homogeneous extension of \(V \restr{\partial Q}\) with respect to the barycenter of \(Q\). Observe that the constant \(\Cr{goodconst21234567} \in (0, \infty)\) depends only on \(m\). The remainder of the proof is similar.
\end{proof}

The case of the singular extension to the ball is slightly more complicated, as we will rely on the parameterization of \(\Bset^{m + 1}_{1}\) by \(\Rset^{m + 1}_+\) through a classical suitable conformal mapping.

\begin{proof}%
[Proof of \cref{theorem_ball}]%
\resetconstant%
We recall that the map 
\begin{equation*}
 \Psi (x) \defeq \frac{4(x + e)}{\abs{x + e}^2} - 2e,
\end{equation*}
where \(e\defeq (0, \dotsc, 0, 1) \in \mathbb{R}^{m+1}_{+}\), defines a diffeomorphism from \(\Bset^{m + 1}_{1}\) to \(\Rset_+^{m + 1}\).
Indeed, if \(x \in \bar{\Bset}^{m+1}_{1} \setminus \{e\}\), then 
\begin{equation*}
 e \cdot \Psi (x) 
 = \frac{2 - 2\abs{x}^2}{\abs{x + e}^2}.
\end{equation*}
Moreover, we have
\begin{equation*}
 \Psi^{-1} (x)
 = \frac{4(x + 2e)}{\abs{x + 2e}^2} - e.
\end{equation*}

In particular, the ball \(\Bset^{m + 1}_{1}\) is isometric to the half-space \(\Rset^{m + 1}_+\) endowed with the metric \(g\) defined for \(x \in \Rset^{m + 1}_+\) and \(v \in \Rset^{m + 1}\) by 
\begin{equation*}
 g (x)[v, v]
 = \frac{16 \abs{v}^2}{\abs{x + 2e}^4}.
\end{equation*}
Moreover, since \(\Psi\) is a conformal map and \(u : \Sset^m \to \manifold{N}\), one has 
\begin{equation*}
 \smashoperator[r]{\iint_{\Sset^m \times \Sset^m}} 
  \frac{d (u(x), u(y))^p}{\abs{x - y}^{2m}} \dif x \dif y 
 = \smashoperator{\iint_{\Rset^m \times \Rset^m}} \frac{d (u(\Psi^{-1} (x)), u(\Psi^{-1}(y)))^p}{\abs{x - y}^{2m}} \dif x \dif y 
\end{equation*}
and 
\begin{equation*}
\smashoperator{\iint_{\substack{x, y \in \Sset^m\\ d(u (x), u (y))\ge \delta}}} \frac{d (u(x), u(y))^p}{\abs{x - y}^{2m}} \dif x \dif y 
= \smashoperator{\iint_{\substack{x, y \in \Rset^m\\ d(u (\Psi^{-1} (x)), u (\Psi^{-1}(y)))\ge \delta}}} \frac{d (u(\Psi^{-1} (x)), u(\Psi^{-1}(y)))^p}{\abs{x - y}^{2m}} \dif x \dif y 
 \end{equation*}
(the reader may also consult \cite{Ahlfors_1981}*{discussion after I,\, (15)}).
We proceed then as in the proof of \cref{theorem_halfspace}, using \(u \compose \Psi^{-1}\) instead of \(u\).
In order to replace the estimate \eqref{eq_cheijo5gee6ahphieQui1sho}, we proceed as follows.
 
Given \(Q \in \mathcal{Q}_{\lambda, \tau, k, h}^+\) and setting 
\begin{align*}
 m_Q &= \inf_{x \in Q} \frac{4}{\abs{x + 2e}^2}&
 &\text{and} &
 M_Q &= \sup_{x \in Q} \frac{4}{\abs{x + 2e}^2},
\end{align*}
we have 
\begin{equation}
\label{eq_maechuquoo1aiQuai6Eig4Ue}
\begin{split}
 \smashoperator[r]{\int_{\substack{x \in Q\\ \abs{\Deriv W (x)} \ge 4t/\abs{x + 2e}^2}}}
 \frac{4^{m + 1}}{\abs{x + 2e}^{2m + 2}} \dif x
 &\le M_Q^{m + 1} \mathcal{L}^{m + 1}\brk{\set{x \in Q \st \abs{\Deriv W (x)} \ge t m_Q}}\\
 &\le \frac{\C}{t^{m + 1}}  \brk[\Big]{\frac{M_Q}{m_Q}}^{m + 1} \tau \lambda^{-k} \int_{\partial Q} |\Deriv V|^{m+1}\\
 &\le \frac{\Cl{constantC2intheorem1}}{t^{m+1}} \brk[\Big]{\frac{M_Q}{m_Q}}^{m + 1} (\tau \lambda^{-k})^{m+1} \sup_{x \in \partial Q} |\Deriv V(x)|^{m+1}\\
 &\le  \frac{\Cr{constantC2intheorem1}(\lambda-1)^{m+1}}{t^{m+1}} \brk[\Big]{\frac{M_Q}{m_Q}}^{m + 1}  \sup_{x \in \partial Q} x^{m+1}_{m+1} |\Deriv V(x)|^{m+1},
\end{split}
\end{equation}
where \(\lambda \geq 2\), \(V\) is defined in \eqref{eq_shee0Ceingooghe2weefei4A} and, up to a suitable bi-Lipschitz homeomorphism between a ball and a cube (see \cite{Griepentrog_Hoppner_Kaiser_Rehberg_2008}*{Cor.~3}), \(W\) is defined on \(Q\) by homogeneous extension of \(V \restr{\partial Q}\) with respect to the barycenter of \(Q\). It is worth noting that the constant \(\Cr{constantC2intheorem1}  \in (0, \infty)\) depends only on \(m\). 
We observe now that if \(x, y \in Q\), then 
\begin{equation}
\label{eq_ugh7ne9aeLieri8iepaiVosh}
 \frac{\abs{x + 2e}}{\abs{y + 2e}}
 \le 1 + \frac{\abs{x - y}}{\abs{y + 2e}}
 \le 1 + \frac{\tau \lambda^{-k}\sqrt{m+1}}{2+\frac{\tau \lambda^{-k}}{\lambda-1}}
 \le 1 + (\lambda - 1)\sqrt{m+1},
\end{equation}
and we have thus by \eqref{eq_maechuquoo1aiQuai6Eig4Ue} and \eqref{eq_ugh7ne9aeLieri8iepaiVosh},
\begin{multline*}
\label{eq_Cie0Eikeid7IeG7Quee4Sam1}
\smashoperator[r]{\int_{\substack{x \in Q\\ \abs{\Deriv W (x)} \ge 4t/\abs{x + 2e}^2}}}
 \frac{4^{m + 1}}{\abs{x + 2e}^{2m + 2}} \dif x\\
 \le  \frac{\Cr{constantC2intheorem1}\brk{\lambda - 1}^{m + 1}}{t^{m + 1}} 
 \brk*{1 + (\lambda - 1)\sqrt{m+1}}^{2m + 2} \sup_{x \in \partial Q} x_{m +1}^{m + 1}\abs{\Deriv V (x)}^{m + 1}.
\end{multline*}
The rest of the proof is similar to the proof of \cref{theorem_halfspace}.
\end{proof}

\end{document}